\documentclass
{amsart}

\allowdisplaybreaks

\usepackage{color}
\usepackage{amsmath}
\usepackage{amssymb}
\usepackage[abbrev]{amsrefs}
\usepackage{stmaryrd}
\usepackage{graphicx}
\usepackage{bm}
\usepackage{xcolor}

\newtheorem{theorem}{Theorem}[section]
\newtheorem{proposition}[theorem]{Proposition}
\newtheorem{lemma}[theorem]{Lemma}

\theoremstyle{definition}
\newtheorem{definition}[theorem]{Definition}

\theoremstyle{definition}
\newtheorem{remark}[theorem]{Remark}

\numberwithin{equation}{section}

\newcommand{\bfone}{{\mathbf{1}}}
\newcommand{\eps}{{\varepsilon}}

\newcommand{\abs}[1]{\left|{#1}\right|}
\newcommand{\norm}[1]{\lVert{#1}\rVert}

\newcommand{\F}{{\mathbb{F}}}
\newcommand{\N}{{\mathbb{N}}}

\newcommand{\R}{{\mathbb{R}}}

\newcommand{\Mean}{{{\mathbb{E}}}}
\newcommand{\Prob}{{{\mathbb{P}}}}

\newcommand{\llb}{{\llbracket}}
\newcommand{\rrb}{{\rrbracket}}

\newcommand{\cF}{{\mathcal{F}}}

\begin{document}
\title[Mean viability]
{Mean viability theorems and\\ 
second-order Hamilton--Jacobi 
equations} 

\author{Christian Keller}
\address
{Department of Mathematics, 
University of  Central Florida,
Orlando, FL 32816, United States}
\thanks{This research  was
 supported in part  by NSF-grant DMS-2106077.}
\email{christian.keller@ucf.edu}


\begin{abstract}
We introduce
 the notion of mean viability for controlled stochastic differential equations 
 and establish counterparts of Nagumo's classical viability theorems 
 (necessary and sufficient 
  conditions for mean viability).
As an application, we  provide  a purely probabilistic proof
of {\color{black} a comparison principle and of}
existence 
  {\color{black} for contingent and viscosity solutions of}
second-order  fully nonlinear path-dependent Hamilton--Jacobi--Bellman 
equations. 
{\color{black}
 We do not use 
  compactness and optimal stopping arguments, which are usually employed 
 in the literature on viscosity solutions for second-order path-dependent  PDEs.} 
\end{abstract}
\maketitle
\pagestyle{plain}

\section{Introduction}

We {\color{black} establish a comparison principle and}
existence 
for second-order Hamilton--Jacobi equations with
 viability-theoretical methods.
To this end, we introduce the notion of mean viability in an appropriate setting.
 Roughly speaking, a pair $(v,K)$ is mean viable
if there is 
a control $a$ such that
$\Mean [v(t,X^a)]\in K$ for all $t\in\R_+$.
Here, $v$ is 
a  function, $K$ 
a set, and $X^a$ 
solves a controlled stochastic differential equation.
Mean viability theorems, which provide necessary and sufficient criteria for mean viability, are main contributions of this paper.
Next, we consider
(path-dependent)  
Hamilton--Jacobi--Bellman (HJB)
equations  of the form
\begin{align}\label{E:Intro:HJB}
&\partial_t u+\inf_{a\in A} \left[b(t,\mathbf{x},a)\,\partial_{\mathbf{x}} u 
+\frac{1}{2}\sigma^2(t,\mathbf{x},a)\, \partial^2_{\mathbf{x}\mathbf{x}} u\right]=0\quad\text{on $[0,T)\times C(\R_+,\R^d)$}
\end{align}
together with  terminal conditions of the form $u(T,\cdot)=h$ on $C(\R_+,\R^d)$.
Generalized 
sub- and supersolutions  are defined 
via certain upper and lower directional derivatives. 
Using our mean viability theorems, 
we {\color{black} prove comparison and} 
 existence 
  for generalized semisolutions.  
{\color{black} Finally, we introduce notions of smooth and  viscosity solutions 
and provide connections to our previously defined generalized  solutions.
This yields a comparison principle for viscosity solutions.} 

\subsection{Background and motivation.} 
There is a large body of research
that applies 
(deterministic) viability theory to 
first-order Hamilton--Jacobi equations  and optimal control (see, e.g., 
\cite{AubinEtAl11Viability,AubinHaddad02PPDE, BettiolVinter17SICON,CannarsaFrankowska91SICON,Carja14SICON,Frankowska89AMO,Frankowska93SICON,Vinter,SubbotinBook}). In particular, existence and uniqueness of generalized  
nonsmooth solutions for Hamilton--Jacobi equations
is established, including for HJB equations related to   optimal control problems with state constraints and discontinuous data.
There are additional   applications in optimal control 
(e.g., optimal synthesis of feedback controls).
Furthermore,  
 viability theory allows for alternative proofs of the comparison
principle for viscosity solutions (avoiding the doubling-the-variables approach). 

While there are also many results on stochastic viability 
(see, e.g., \cite{AubinDaPrato90, BardiGoatin99, BardiJensen02, BFQ, BuckdahnEtAl98, Goreac11}), there have been less applications compared to the deterministic case.
In particular, it is not clear if the existing results on stochastic viability can be used to show uniqueness for second-order HJB equations. Purpose of this work is to remedy this situation,
 i.e., to obtain counterparts for  at least some of the mentioned deterministic results above
by working with a different probabilistic notion of viability: Mean viability instead of stochastic viability.


\subsection{Stochastic and mean viability: A short overview and our contributions}
For simplicity, we consider  the Markovian case in this subsection.
Note that the rest of this paper considers the non-Markovian case.

We recall  the standard notions of stochastic viability from \cite{BardiJensen02,BFQ,BuckdahnEtAl98}.
Consider a controlled stochastic differential equation of the form
\begin{equation}\label{E:MarkovSDE}
\begin{split}
dX^{t,x,a}_s&=b(s,X^{t,x,a}_s,a_s)\,ds+\sigma(s,X^{t,x,a}_s,a_s)\,dW_s\text{ on $[t,\infty)$,}\\
X^{t,x,a}_t&=x\in\R^d.
\end{split}
\end{equation}
A set $K\subset\R^d$ is \emph{viable} for \eqref{E:MarkovSDE} if, whenever $x\in K$, there is a control $a$
such that $X^{t,x,a}_s$ stays in $K$ for all $s\ge t$. Similarly, $K$ is \emph{$\eps$-viable} for \eqref{E:MarkovSDE} if, whenever $x\in K$, there is some constant
$c>0$ such that,
for every $\eps>0$,
there is a control $a$ with
\begin{align*}
\Mean \left[\int_t^\infty e^{-cs} d^2_K(X^{t,x,a}_s)\,ds\right]<\eps.
\end{align*}
 Here, $d_K(x)$ denotes the distance from a point $x$ to the set $K$. Note that for $\eps$-viability,
one still wants  the \emph{trajectories of  $X^{t,x,a}$} to be close to $K$ in some sense.

Here, we consider 
another notion of viability, which formally subsumes stochastic viability.
Given a function $v:\R_+\times\R^d\to\R$ and a set $K\subset\R$, we call the pair $(v,K)$ \emph{mean viable} for
\eqref{E:MarkovSDE} if, whenever $v(t,x)\in K$, there is a control $a$ such that
\begin{align}\label{E:General:Case}
\Mean\left[v(s,X^{t,x,a}_s)\right]\in K\text{ for all $s\ge t$.}
\end{align}

To obtain applications for second-order Hamilton--Jacobi equations in the 
spirit 
of the first-oder case (e.g., as done in \cite{Frankowska93SICON}), we focus on
a suitable stochastic counterpart 
 of the \emph{deterministic viability} of the epigraph 
of $v$,
i.e., if $y\ge v(t,x)$, then there is 
a trajectory $(x(\cdot),y(\cdot))$ starting at $(x,y)$ at time $t$ that satisfies $y(s)\ge v(s,x(s))$ for all $s\ge t$.
Hence,  we  consider the following special case of \eqref{E:General:Case}, for which 
we establish sufficient and necessary tangential conditions:
\begin{align*}
\Mean\left[\widehat{v}(s,X_s^{t,x,a},y)\right]\in K\text{ for all $s\ge t$},
\end{align*}
where $\widehat{v}$ is of the form $\widehat{v}(t,x,y)=v(t,x)-y$ and $K=(-\infty,0]$. 
This corresponds, up to some 
technical details, to the notion of \emph{stochastic $u$-stability} in
\cite{SubbotinaEtAl85,Subbotina06}, where infinitesimal necessary  and sufficient criteria are established as well.
Note that the diffusion term in \cite{SubbotinaEtAl85,Subbotina06} can only be time-dependent, i.e., of the form $\sigma=\sigma(t)$.
Our diffusion term can also be  state- and control-dependent, i.e., of the form $\sigma=\sigma(t,x,a)$.
Moreover, (after this subsection) our setting is non-Markovian whereas the setting
in \cite{SubbotinaEtAl85,Subbotina06} is Markovian.


\subsection{Contributions to partial differential equations}
The probabilistic  approach in this paper provides a
third method of proving the  comparison principle for a fairly general class of possibly degenerate parabolic HJB equations
 besides the standard viscosity approach (see \cite{UserGuide}) 
 and the 
  approximation arguments used in \cite{Lions83HJB2} and many works in the path-dependent case (see below).
  {\color{black} Note that a probabilistic approach different from ours has been used in \cite{RTZ20SICON}
  albeit only in the semilinear case.}

Our notion of   generalized 
solutions for \eqref{E:Intro:HJB}  in Section~\ref{S:HJB} is motivated by  minimax solutions for 
Hamilton--Jacobi equations 
of the form $\partial_t v+\frac{1}{2} \sigma^2(t)\,\partial_{xx}^2 v + H(t,x,\partial_x v)=0$ on
$[0,T)\times\R^d$ in \cite{SubbotinaEtAl85,Subbotina06}  (see also \cite{BK22SICON} for the path-dependent case).
Our  main contribution  
 compared  to \cite{BK22SICON,SubbotinaEtAl85,Subbotina06} 
  is the incorporation of the controlled volatility case.

 Our notion of viscosity solutions in Subsection~\ref{SS:ViscSol} is in the spirit of \cite{cosso18,EKTZ11,ETZ_I}, i.e.,
 test functions  are required to be  tangent in mean
to the candidate solution.  
Thereby, the spaces of test functions are enlarged, which makes it, in principle, easier to prove uniqueness but more difficult 
to prove existence.

Note that our notion of viscosity solutions is relatively specific (only HJB equations of
the form \eqref{E:Intro:HJB} are covered). 
However,  our comparison principle (Theorem~\ref{T:viscosity:comparison})
is  stronger than any other comparison principle for path-dependent fully nonlinear second-order HJB equations
in the literature in the following sense: \cite{zhou2020viscosity} requires continuity of viscosity sub- and supersolution
(see, in particular, Remark 6.7 therein) 
whereas we allow
for semicontinuity in time; 
\cite{cosso21v2path}
 requires a stronger condition for the controlled diffusion term
(the additional condition (C) therein); 
\cite{ETZ_II}
 requires first,  strong uniform continuity assumptions 
 (which means that  $\sigma(\cdot,\cdot,\cdot)$ needs to be of the form $\sigma(t,\mathbf{x},a)=\sigma(a)$),
and second, uniform 
ellipticity whereas we need just (A1) and (A2) from Section~\ref{S:Data};
\cite{RR20SIMA,RTZ17}  require
a (stronger) $\norm{\cdot}_{L^p}$-continuity 
compared to our $\norm{\cdot}_\infty$-continuity;
\cite{EkrenZhang16PUQR} operates in a  ``piece-wise Markovian" setting, which
requires stronger conditions for existence. 

Moreover, note that the comparison principle for fully nonlinear HJB equations in \cite{cosso18} was left as
an open problem. Comparison was established therein for equations    on $[0,T)\times C([0,T],H)$, $H$ being a Hilbert space,
that are \emph{linear in the derivatives} and also involve an unbounded operator on $H$. 
Since our definition of viscosity solutions is substantially the same as in \cite{cosso18}
and since we (like \cite{cosso18} in the linear case)
 do not rely on compactness arguments and also do not use the Crandall--Ishii lemma,
our method gives new hope to attack the  mentioned open problem in \cite{cosso18}
(see also p.~364 in \cite{FabbriGozziSwiech} for further related  discussion).



\subsection{Organization of the rest of the paper}
In Section~\ref{S:Notation}, we introduce general notation.
Section~\ref{S:Data} contains the data for our problem and the standing assumptions.
Topic of Section~\ref{S:MeanViability} is mean viability:  The notion itself 
and the mean viability theorems, i.e., necessary and sufficient conditions for mean viability. 
In Section~\ref{S:HJB}, the mean viability theorems are
applied to second-order HJB equations. We prove existence
and 
{\color{black} a comparison principle}
 for so-called quasi-contingent solutions. {\color{black} In
Section~\ref{S:ClassicalAndViscosity}, we introduce classical and viscosity solutions for our
HJB equations and establish connections to quasi-contingent solutions, which provides us with
 a comparison principle for viscosity solutions. 
}
 Section~\ref{Section:Step3} 
 contains 
  technical material
  needed for the proof of 
  one of the mean viability theorems.
   
\section{Notation}\label{S:Notation}
Let $\N=\{1,2,\ldots\}$, $\overline{\R}=\R\cup\{\pm\infty\}$, and $\R_+=[0,\infty)$. 
Let $\mathcal{B}(E)$ be 
the Borel $\sigma$\--field of a to\-po\-lo\-gi\-cal space $E$ 
and $\mathcal{P}(E)$ the set of 
probability measures on $\mathcal{B}(E)$.

Fix $d\in\N$. Let $\Omega=C(\R_+,\R^d)$, $\Prob$ be the Wiener measure on
$\mathcal{B}(\Omega)$, 
$\Mean=\Mean^\Prob$ the corresponding expectation,
 $W=(W_t)_{t\ge 0}$ defined by $W_t(\omega)=\omega(t)$ the
canonical process, and $\F=(\cF_t)_{t\ge 0}$ defined by $\cF_t=\sigma(W_s:\,0\le s\le t)$ the 
canonical filtration on $\Omega$. We also consider the shifted  filtrations $\F^t=(\cF^t_s)_{s\ge t}$
defined by $\cF^t_s=\sigma(W_r-W_t:\, t\le r\le s)$ with their progressive
$\sigma$-fields $\mathrm{Prog}(\F^t)$.

Given a normed space $E$, a $\sigma$-field $\mathcal{G}$, and $t\in\R_+$,
we use the spaces 
$\mathbb{L}^2(\mathcal{G},\Prob,E)$ of all $\mathcal{G}$-measurable random variables
$\xi:\Omega\to E$ with $\Mean [\abs{\xi}^2]<\infty$
and
\begin{align*}
\mathbb{L}^0(\mathbb{F}^t, E)&:=
\{\text{all $\F^t$-progressive processes from $\R_+\times\Omega$ to $E$}\},\\
\mathbb{L}^2(\mathbb{F}^t,\Prob,E)&:=
\{ X\in \mathbb{L}^0(\mathbb{F}^t, E):\, 
\norm{X}_{\mathbb{L}^2(\mathbb{F}^t,\Prob,E)}^2:=
\Mean\left[\int_t^\infty \abs{X_s}^2\,ds\right]<\infty\},\\
\mathbb{S}^2(\mathbb{F}^t,\Prob,E)&:=
\{\text{all $\Prob$-a.s.~continuous }
X\in\mathbb{L}^0(\mathbb{F}^t, E)
\text{ with $\Mean [
\norm{X}_\infty^2]<\infty$}\},
\end{align*}
where $\cF^t_s$ is implicitly understood to be equal to $\{\emptyset,\Omega\}$  whenever $s<t$
and $\norm{\cdot}_\infty$ is the usual sup norm (cf.~\cite{Zhang17book} for much of our notation).

The transpose of a matrix $q$ is denoted by $q^\top$. Given real matrices $q$ and $\tilde{q}$
of appropriate dimensions, we write $q:\tilde{q}$ for the trace of $q \tilde{q}^\top$. 
The identity matrix in $\R^{d\times d}$ is denoted by $I_{d\times d}$.
We always write $0$ for any zero vector. 

We write  $\mathbf{1}$ for  indicator functions, i.e., given a set $E$, $\mathbf{1}_E(x)=1$ if $x\in E$ and
$\mathbf{1}_E(x)=0$ otherwise. {\color{black} Dirac measures are denoted by $\delta_x$.}

We use the usual notation for stochastic intervals (Definition~60 in \cite{DM}), e.g., 
for $\F$-stopping times $\tau_1$ and $\tau_2$, 
$\llb \tau_1, \tau_2 \rrb=\{(t,\omega)\in\R_+\times\Omega:\, \tau_1(\omega)\le t\le \tau_2(\omega)\}$.

Given $s$, $t\in\R$, put $s\vee t:=\max\{s,t\}$ and $s\wedge t:=\min\{s,t\}$.

We write l.s.a.~for lower semianalytic, l.s.c.~for lower semicontinuous, and u.s.c.~for upper semicontinuous.

Non-empty subsets of $\R_+\times\Omega$ are always understood as pseudo-metric spaces
equipped with the pseudo-metric $((t,\mathbf{x}),(t^\prime,\mathbf{x}^\prime))\mapsto
\abs{t-t^\prime}+\norm{\mathbf{x}_{\cdot\wedge t}-\mathbf{x}^\prime_{\cdot\wedge t^\prime}}_\infty$.

\section{Data and standing assumptions}\label{S:Data}
Data
are a 
{\color{black} Borel subset $A$ of $[0,1]$},\footnote{
This is not restrictive (see, e.g., p.~25 in \cite{ElKarouiTanII}) but convenient for technical reasons.} 
{\color{black} which is used as action space,}
and 
functions
\begin{align}\label{E:Data}
b:\R_+\times\Omega\times A\to\R^d,\,\sigma:\R_+\times\Omega\times A\to\R^{d\times d},\,h:\Omega\to\R.
\end{align}
If there is danger of ambiguity, then we shall write $b(\cdot,\cdot,\cdot)$, etc.~instead of $b$, etc. Moreover, we fix a triple
\begin{align}\label{E:GlobalTriple}
(a^\circ,p^\circ,q^\circ)\in A\times\R^d\times\R^{d\times d}.
\end{align}

As 
 {\color{black} set of admissible controls}, we use
\begin{align*}
\mathcal{A}:=\{a:\R_+\times\Omega\to A\quad\text{$\F$-predictable}\}.
\end{align*}
Note that $\mathcal{A}$  {\color{black} equipped with the metric given by 
$(\Mean[\int_0^\infty e^{-t} \abs{a_t-\tilde{a}_t}^2\,dt])^{1/2}$, $a$, $\tilde{a}\in\mathcal{A}$,} is a Polish space (see p.~23 in \cite{ElKarouiTanII} or p.~823 in
\cite{Djete22AOP}). We also use 
\begin{align*}
\mathcal{A}^t:=\{a\in\mathcal{A}: 
\text{$a$ is $\F^t$-predictable 
  with $a\vert_{\llb 0,t\rrb}=
 a^\circ$}\},\quad t\in [0,T].
\end{align*}
Similarly, we  assume that the restriction of any $(p,q)\in\mathbb{L}^2(\F^t,\Prob,\R^d)\times\mathbb{L}^2(\F^t,\Prob,\R^{d\times d})$ to $\llb 0,t\rrb$ 
coincides with the constant process $(s,\omega)\mapsto (p^\circ,q^\circ)$.

{\color{black} We fix $T\in (0,\infty)$ 
 and a  universally measurable  
 function
 $v:[0,T]\times\Omega\to\R$
that is non-anticipating, i.e.,
 $v(t,\mathbf{x})=v(t,\mathbf{x}_{\cdot\wedge t})$.}
{\color{black}
Also set
\begin{equation}\label{E:structural}
\begin{split}
\widehat{v}(t,\mathbf{x},y)&:=v(t,\mathbf{x})-y,\quad (t,\mathbf{x},y)\in [0,T]\times\Omega\times\R,\text{ and}\\
K&:=(-\infty,0].
\end{split}
\end{equation}}

The following standing assumptions are always in force:

(A1) The function $h$ is  $\cF_T$-measurable
 and the functions $b(\cdot,\cdot,\cdot)$ and $\sigma(\cdot,\cdot,\cdot)$ are Borel measurable, they vanish after time $T$, and  they are non-anticipating, i.e.,
for every $(t,\mathbf{x},a)\in \R_+\times\Omega\times A$,
\begin{align*}
(b,\sigma)(t,\mathbf{x},a)=\bfone_{[0,T]}(t).(b,\sigma)(t,\mathbf{x}_{\cdot\wedge t},a)\quad\text{and}\quad
h(\mathbf{x})=h(\mathbf{x}_{\cdot \wedge T}).
\end{align*}

(A2) There is a constant $L_b\ge 1$ such that,
 for all $(t,\mathbf{x},\tilde{\mathbf{x}},a)\in \R_+\times\Omega\times\Omega\times A$,
 \begin{align*}
\abs{b(t,\mathbf{x},a)-b(t,\tilde{\mathbf{x}},a)}+\abs{\sigma(t,\mathbf{x},a)-\sigma(t,\tilde{\mathbf{x}},a)}&\le 
L_b\norm{\mathbf{x}_{\cdot\wedge t}-\tilde{\mathbf{x}}_{\cdot\wedge t}}_\infty\quad\text{and}\\
\abs{b(t,\mathbf{x},a)}+\abs{\sigma(t,\mathbf{x},a)}
+
\abs{h(\mathbf{x})}&\le L_b.
\end{align*}


The following hypothesis 
will also be used: 

(H) The function $v$ is bounded from below and there is a constant $L\ge 1$ such that, for all $\mathbf{x}$, $\tilde{\mathbf{x}}\in\Omega$ and $s$, $t\in\R_+$ with $s<t$,
\begin{align*}
\abs{v(t,\mathbf{x})-v(t,\tilde{\mathbf{x}})}&\le L\norm{\mathbf{x}_{\cdot\wedge t}-\tilde{\mathbf{x}}_{\cdot\wedge t}}_\infty\text{ and }\\
\abs{v(t,\mathbf{x}_{\cdot\wedge s})-v(s,\mathbf{x})}&\le L(1+\norm{\mathbf{x}_{\cdot\wedge s}}_\infty)(t-s)^{1/2}. 
\end{align*}


\section{Mean viability}\label{S:MeanViability}

\subsection{The notions}
{\color{black} Given  $a\in\mathcal{A}$ and $(t,\mathbf{x})\in\R_+\times\Omega$, consider} 
 the controlled stochastic differential  
 {\color{black} equation}
\begin{equation}\label{E:mv:CSDE}
\begin{split}
dX_s^{t,\mathbf{x},a}&=b(s,X^{t,\mathbf{x},a},a_s)\,ds+\sigma(s,X^{t,\mathbf{x},a},a_s)\,dW_s
\quad\text{on $(t,\infty)$, $\Prob$-a.s.,}\\
X^{t,\mathbf{x},a}\vert_{[0,t]}&=\mathbf{x}\vert_{[0,t]},
\end{split}
\end{equation}
{\color{black} which, under  (A1) and (A2),  has a unique (strong)  solution in $\mathbb{S}^2(\F^t,\Prob,\R^d)$ 
(this is a special case of Proposition~2.8 in \cite{CossoEtAl21McKeanVlasov}; alternatively one can
slightly modify the proof of Theorem~3.3.1 in \cite{Zhang17book} to cover path-dependent coefficients).}

 
{\color{black}  We introduce now  notions of mean viability for the pair $(\widehat{v},K)$ defined by \eqref{E:structural}
(see~ \cite{Goreac11} for  corresponding stochastic viability notions).}

\begin{definition} \label{D:MeanViab}
The pair $(\widehat{v},K)$ is \emph{mean viable} for \eqref{E:mv:CSDE} if, for every
 $(t,\mathbf{x},y)\in [0,T]\times\Omega\times\R$ {\color{black} with $v(t,\mathbf{x})\le y$,}
there  is 
an  
$a\in\mathcal{A}^t$  such that, for every $s\in [t,T]$,
\begin{align*}
{\color{black} \Mean\left[v(s,X^{t,\mathbf{x},a})\right]\le y.
}
\end{align*}
\end{definition}




\begin{definition}\label{D:WeakApproxViab} 
The
 pair $(\widehat{v},K)$ is \emph{approximately mean viable}  for \eqref{E:mv:CSDE} if, for each
 $(t,\mathbf{x},y)\in [0,T]\times\Omega\times\R$ 
 with {\color{black}  $v(t,\mathbf{x})\le y$,}  
 {\color{black} we have}
\begin{align}\label{E:WeakApproxViabDef}
{\color{black}
 \inf_{a\in\mathcal{A}^t}\sup_{s\in [t,T]} \Mean\left[v(s,X^{t,\mathbf{x},a})\right]\le y.}
\end{align}
\end{definition}
{\color{black}\begin{remark}
These notions can  be formulated for more general pairs $(\widehat{v},K)$, e.g.,  mean viability means  that, if $\widehat{v}(t,\mathbf{x},y)\in K$, then
there is an $a\in\mathcal{A}^t$ such that $\Mean\left[\widehat{v}(s,X^{t,\mathbf{x},a},y)\right]\in K$ for all $s\in [t,T]$,
but \eqref{E:structural} is crucial for  establishing our sufficient condition for mean viability (see  Remark~\ref{R:structural} for an explanation).
\end{remark}}

{\color{black}
\begin{remark}
There is a natural connection between stochastic viability and stochastic target problems (see Section~7 in \cite{ST02JEMS}).
Similarly, mean viability is related to stochastic target problems with expected loss (cf.~\cite{BouchardEtAl09}), in particular,
to superhedging  of American contingent claims subject to  expected loss constraints.
\end{remark}}

\subsection{Mean tangency} \label{S:MeanTangency}
The following 
definition
 is motivated by a deterministic counterpart 
 in \cite{Carja09TAMS}~and its stochastic
version in \cite{Goreac11}.

\begin{definition}\label{D:QTS}
Fix $(t,\mathbf{x},y)\in [0,T]\times\Omega\times\R$ {\color{black} with $v(t,\mathbf{x})\le y$}. 
A non-empty set
$\mathcal{E}\subset
\mathbb{L}^2(\F^t,\Prob,\R^d)\times\mathbb{L}^2(\F^t,\Prob,\R^{d\times d})$ is 
 \emph{mean quasi-tangent} to $(\widehat{v},K)$ at $(t,\mathbf{x},y
 )$ if, for each
 $\eps>0$, there 
 {\color{black} is} a 
$(\delta,(b,\sigma),p,q)\in 
{\color{black} (0,\eps]}
\times\mathcal{E}\times\mathbb{L}^0(\F^t,\R^d)
\times\mathbb{L}^0(\F^t,\R^{d\times d})$ 
with
\begin{align}\label{E:D:QTS:pq}
\norm{\bfone_{\llb t,t+\delta\rrb}.p)}^2_{\mathbb{L}^2(\F^t,\Prob,\R^d)}+
\norm{\bfone_{\llb t,t+\delta\rrb}.q}^2_{\mathbb{L}^2(\F^t,\Prob,\R^{d\times d})}\le \eps\delta
\end{align}
and 
 {\color{black}
 $\Mean\left[v(t+\delta,X^{t,\mathbf{x};b,\sigma;p,q})
\right]\le y+\eps\delta$}
where
    $X=X^{t,\mathbf{x};b,\sigma;p,q}$  satisfies 
\begin{equation}\label{E:D:QTS:X}
\begin{split}
X_s&=\mathbf{x}_t+\int_t^{s} [b_r+p_r]\,dr
+\int_t^{s} [\sigma_r+q_r]\,dW_r\,\,\,\text{on $(t,t+\delta]$, $\Prob$-a.s.,}\\
  X_s&=\mathbf{x}_s\quad\text{ on $[0,t]$.}
\end{split}
\end{equation}

We write $\mathcal{QTS}_{\widehat{v},K}(t,\mathbf{x},y
)$ for the \emph{class of all mean quasi-tangent sets} 
to $(\widehat{v},K)$ at $(t,\mathbf{x},y
)$.
\end{definition}

\subsection{Necessary and sufficient conditions for mean viability}
Now, we can state appropriate versions of the classical viability theorems. 
To this end, consider
\begin{equation}\label{E:E+}
\begin{split}
&\mathcal{E}_+(t,\mathbf{x}):=
\{({b},{\sigma})\in\mathbb{L}^2(\F^t,\Prob,\R^d)\times\mathbb{L}^2(\F^t,\Prob,\R^{d\times d}):\,\exists a\in\mathcal{A}^t:\\
&\qquad
 (b,\sigma)(s,\omega)=(b,\sigma)(s,\mathbf{x}_{\cdot\wedge t},a(s,\omega))
 \quad\text{$dt\times d\Prob$-a.e.~on $[t,T]\times\Omega$}\}, 
 \end{split}
 \end{equation}
 where $(t,\mathbf{x})\in [0,T]\times\Omega$
 (we write $\mathcal{E}_+$ to indicate that this set depends on the values of 
 $s\mapsto (b,\sigma)(s,\cdot,\cdot)$ from \eqref{E:Data} in a right neighborhood of $t$).

\begin{theorem}\label{T:Necc}
 Suppose that  $(\widehat{v},K)$ is mean viable or approximately mean viable  for \eqref{E:mv:CSDE}.
Then, for each $(t,\mathbf{x},y)\in [0,T)\times\Omega\times\R$ with 
{\color{black} $v(t,\mathbf{x})\le y$, }
we have
\begin{align}\label{E:Necc}
\mathcal{E}_+(t,\mathbf{x})\in  \mathcal{QTS}_{\widehat{v},K}(t,\mathbf{x},y).
\end{align}
\end{theorem}

\begin{proof}
We 
consider  the case that 
$(\widehat{v},K)$ is approximately mean viable for \eqref{E:mv:CSDE}.
Fix $\eps>0$, $(t,\mathbf{x},y)\in [0,T)\times\Omega\times\R$ with 
{\color{black} $v(t,\mathbf{x})\le y$, and}  $\delta\in(0,\eps\wedge (T-t)]$.
 {\color{black} Then there is an $a\in\mathcal{A}^t$} such
that 
{\color{black} $\Mean\left[v(t+\delta,X^{t,\mathbf{x},a})\right]\le y+\eps\delta$.}
Note that, by (A2) and standard estimates (see, e.g., Section~3.2 in \cite{Zhang17book}),
\begin{align}\label{E:Necc:Proof}
\Mean\left[\norm{X^{t,\mathbf{x},a}_{\cdot\wedge(t+\delta)}-\mathbf{x}_{\cdot\wedge t}}_\infty^2\right]\le CL_b^2\,\delta
\end{align}
for some $C>0$ independent of $\delta$ and $a$.
Next, we need to find 
 $(b,\sigma)\in\mathcal{E}_+(t,\mathbf{x})$ and 
 $(p,q)\in{\color{black}\mathbb{L}^0(\F^t,\R^d)\times\mathbb{L}^0(\F^t,\R^{d\times d})}$
such that \eqref{E:D:QTS:pq}  holds 
 and
 \begin{equation}\label{E00:VS:existence}
 b(s,X^{t,\mathbf{x},a},a_s)=b_s+p_s\quad\text{and}\quad  \sigma(s,X^{t,\mathbf{x},a},a_s)=\sigma_s+q_s.
 \end{equation}
 We take $(b,\sigma)$ defined by $(b,\sigma)(s,\mathbf{x}_{\cdot\wedge t},a_s)$ and $(p,q)$ implicitly defined
 by \eqref{E00:VS:existence}. 
 Then 
  {\color{black}
 $\Mean\left[v(t+\delta,X)
\right]\le y+\eps\delta$}
with $X$ defined by  \eqref{E:D:QTS:X} and,
 from (A2) and \eqref{E:Necc:Proof}, we can deduce that \eqref{E:D:QTS:pq}  holds
 if $\delta\le (2CL_b^4)^{-1}\eps$, i.e., we have \eqref{E:Necc}.
\end{proof}

 
 {\color{black}
 \begin{remark}\label{R:structural}
 We 
 need 
 the pair $(\widehat{v},K)$ to satisfy 
 \eqref{E:structural} 
 for proving the following main result. The reason  
  is that, in general, mean viability
does not propagate in time, i.e.,  
$X_t\in K\Rightarrow \Mean[X_{t+\delta}]\in K$ 
and
$X_{t+\delta}\in K\Rightarrow \Mean[X_{t+2\delta}]\in K$
do
not, in general, imply
$ X_t\in K\Rightarrow  
\Mean[X_{t+2\delta}]\in K$.
 This is in contrast to  (stochastic) viability, where 
 $X_t\in K\Rightarrow X_{t+\delta}\in K$
 and
 $X_{t+\delta}\in K\Rightarrow X_{t+2\delta}\in K$
 imply
 $X_t\in K\Rightarrow  X_{t+2\delta}\in K$.
The requirement \eqref{E:structural} circumvents this issue explicitly  in
 \eqref{Es:FinalInequality}.
 \end{remark}
 }


\begin{theorem}\label{T:Sufficient}
{\color{black} Let 
hypothesis~(H) hold.} 
 If, 
for each $(t,\mathbf{x},y)\in [0,T)\times\Omega\times\R$
with 
{\color{black} $v(t,\mathbf{x})\le y$, }
we have
$\mathcal{E}_+(t,\mathbf{x})\in
\mathcal{QTS}_{\widehat{v},K}(t,\mathbf{x},y)$,
then  $(\widehat{v},K)$ is approximately mean viable for \eqref{E:mv:CSDE}. 
\end{theorem}

The proof of this theorem is deferred to the end of Section~\ref{S:Approx}.

\begin{remark} \label{R:Sufficient:Weaker}
In Theorem~\ref{T:Sufficient}, one can drop the time regularity assumption in hypothesis (H). 
But then the conclusion becomes  weaker. Instead of approximate mean viability we would have the property
that, for each
 $(t,\mathbf{x},y)\in [0,T]\times\Omega\times\R$ 
 with  {\color{black} $v(t,\mathbf{x})\le y$} 
 and for each $s\in [t,T]$,
 {\color{black} we have
 $\inf_{a\in\mathcal{A}^t} \Mean\left[v(s,X^{t,\mathbf{x},a}) 
\right]\le y$.}
\end{remark}

\subsection{Approximate solutions}\label{S:Approx}
We introduce now  approximate solutions of \eqref{E:mv:CSDE}.
Existence of those solutions is crucial for the proof of Theorem~\ref{T:Sufficient}.
Our definition is an appropriate adaption of Definition~3 in
\cite{Goreac11} to our setting and the next proof follows the structure of the proof of Lemma~1
in \cite{Goreac11}. However, details are different and relatively heavy technical arguments are needed (see especially Section~6).

\begin{definition}\label{D:Approx}
Fix $\eps>0$ and $(t,\mathbf{x},y)\in [0,T)\times\Omega\times\R$.
An \emph{$\eps$-approximate solution 
 of \eqref{E:mv:CSDE} for $(\widehat{v},K)$ starting at $(t,\mathbf{x},y)$}   
is a sextuple  $(\tau,\varrho,a,p,q,X)$ such that the following holds:

 (i) $\tau$ is an $\F$-stopping time with $t<\tau\le T$,
 
(ii) $\varrho=\varrho(s,\omega):\R_+\times\Omega\to\R_+$
is non-decreasing  and c\`adl\`ag in $s$, $\F$-adapted,
and  we have $s-\eps\le \varrho_s\le s$ on $\llb 0,\tau\llb$
and $\varrho_s=s$ on $\llb \tau,\infty\llb$, 

(iii) $(a,p,q,X)\in\mathcal{A}^t\times \mathbb{L}^0(\F^t,\R^d)\times\mathbb{L}^0(\F^t,\R^{d\times d})\times
\mathbb{S}^2(\F,\Prob,\R^d)$ such that
\begin{align}\label{E:tptq:eps}
\norm{\bfone_{\llb t,\tau\rrb}.p)}^2_{\mathbb{L}^2(\F^t,\Prob,\R^d)}+
\norm{\bfone_{\llb t,\tau\rrb}.q}^2_{\mathbb{L}^2(\F^t,\Prob,\R^{d\times d})}\le \eps\cdot\Mean[\tau-t],
\end{align}

 (iv) $(\bfone_{\llb t,\tau\rrb}(s,W).(b,\sigma)(s,X_{\cdot\wedge\varrho_s},a_s))_{s\ge 0}\in 
 \mathbb{L}^2(\F,\Prob,\R^d)\times\mathbb{L}^2(\F,\Prob,\R^{d\times d})$,
 
 (v)  $\Prob$-a.s.
for every $s\ge t$, 
\begin{equation*}
\begin{split}
X_{s\wedge\tau}&=\mathbf{x}_t+\int_t^{s\wedge\tau}\left[b(r,X_{\cdot\wedge\varrho_r}
,a_r)+p_r\right]\,dr
+\int_t^{s\wedge\tau}\left[\sigma(r,X_{\cdot\wedge\varrho_r},a_r)+q_r\right]\,dW_r,\\
X\vert_{[0,t]}&=\mathbf{x}\vert_{[0,t]},
\end{split}
\end{equation*}

 (vi)  
 {\color{black}
 $\Mean\left[v(\varrho_s\wedge\tau,X)\right]\le y+ \eps\cdot \Mean[s\wedge\tau-t]$ for all $s\in [t,T]$.}
 \end{definition}

\begin{proposition}\label{P:epsApproxSol}
Let $v$ be l.s.c.~and bounded from below 
and $T_1\in (0,T]$.
{\color{black} If,}
for each $(t,\mathbf{x},y)\in [0,T_1)\times\Omega\times\R$ with 
{\color{black} $v(t,\mathbf{x})\le y$,}
we have
$\mathcal{E}_+(t,\mathbf{x})\in
\mathcal{QTS}_{\widehat{v},K}(t,\mathbf{x},y)$,
{\color{black} then,}
for each $\eps>0$ and $(t,\mathbf{x},y)\in [0,T_1)\times\Omega\times\R$  with 
{\color{black} $v(t,\mathbf{x})\le y$,}
there is 
 an $\eps$-appro\-xi\-mate so\-lu\-tion
$(\tau,\varrho,a,p,q,X)$ of  \eqref{E:mv:CSDE} for $(\widehat{v},K)$  starting at $(t,\mathbf{x},y)$
 such that $\tau= T_1$, $\Prob$-a.s.
\end{proposition}
\begin{proof}
Fix $\eps>0$ and $(t,\mathbf{x},y)\in [0,T)\times\Omega\times\R$ with $y\ge v(t,\mathbf{x})$.
 We consider only the case $T_1=T$.

\textit{Step~1} (local existence for some $\tau$). 
By Definition~\ref{D:QTS} and \eqref{E:E+}, the relation
$\mathcal{E}_+(t,\mathbf{x})\in\mathcal{QTS}_{\widehat{v},K}(t,\mathbf{x},y)$ 
 yields the existence of a quintuple 
\begin{align*}
(\delta,{X},{a},{p},{q})\in 
 (0,\eps\wedge (T-t)]\times\mathbb{S}^2(\F,\Prob,\R^d)\times
\mathcal{A}^t\times \mathbb{L}^0(\F^t,\R^d)\times\mathbb{L}^0(\F^t,\R^{d\times d})
\end{align*}
such that \eqref{E:D:QTS:pq}, \eqref{E:D:QTS:X} with $(b,\sigma)\in\mathcal{E}_+(t,\mathbf{x})$,
and {\color{black} $\Mean[v(t+\delta,X)]\le\ y+\eps\delta$}
hold. 
Define $\varrho:\R_+\times\Omega\to\R_+$  by
$\varrho_s:=s.\bfone_{[0,t)\cup [t+\delta,\infty)}(s)+t.\bfone_{[t,t+\delta)}(s)$
and $\tau:=t+\delta$. Then
$(\tau,\varrho,a,p,q,X)$
is an
 $\eps$-approximate solution
 of \eqref{E:mv:CSDE} for $(\widehat{v},K)$ starting at $(t,\mathbf{x},y)$.

\textit{Step~2} (existence for a maximal $\tau$). 
Denote by $\mathcal{S}$ the set of all $\eps$-approximate solutions of \eqref{E:mv:CSDE}
  for $(\widehat{v},K)$ starting at $(t,\mathbf{x},y)$.
We equip $\mathcal{S}$ with a preorder 
 $\preceq$ 
  defined as follows: For two elements 
  $\mathfrak{s}^{\text{($i$)}}=(\tau^{(i)},\varrho^{(i)},a^{(i)},p^{(i)},q^{(i)},X^{(i)})$,
  $i=1$, $2$, in
  $\mathcal{S}$, we have $\mathfrak{s}^{\text{($1$)}}\preceq \mathfrak{s}^{\text{($2$)}}$
if  $\tau^{\text{($1$)}}\le\tau^{\text{($2$)}}$ 
and, outside of some $\Prob$-evanescent set, 
$(\varrho^{(2)},a^{(2)},p^{(2)},q^{(2)})
(\cdot\wedge\tau^{\text{($1$)}})
=(\varrho^{(1)},a^{(1)},p^{(1)},q^{(1)})
(\cdot\wedge\tau^{\text{($1$)}})$.
Also, define a function 
\begin{align*}
\mathcal{N}:\mathcal{S}\to [t,T],\quad  \mathfrak{s}=(\tau,\varrho,a,p,q,X)\mapsto \mathcal{N}(\mathfrak{s}):=\Mean[\tau].
\end{align*}

We will invoke the Brezis-Browder ordering principle (see, e.g., Theorem~2.1.1 in 
\cite{CarjaVrabie05Handbook}),
which will provide us with an $\mathcal{N}$-maximal element $\mathfrak{s}_0$ of $\mathcal{S}$, i.e., 
$\mathfrak{s}\in\mathcal{S}$ and $\mathfrak{s}_0\preceq\mathfrak{s}$  imply
$\mathcal{N}(\mathfrak{s}_0)=\mathcal{N}(\mathfrak{s})$. 
Since $\mathcal{N}$ is increasing, 
it suffices to verify that
each increasing sequence in $\mathcal{S}$ has an upper bound in $\mathcal{S}$. 
To this end,
fix an increasing sequence 
$(\mathfrak{s}^{\text{($n$)}})_{n\ge 1}=
(\tau^{(n)},\varrho^{(n)},a^{(n)},p^{(n)},q^{(n)},X^{(n)})_{n\ge 1}$
 in $\mathcal{S}$. Put 
  $\bar{\tau}:=\sup_{n\ge 1} \tau^{\text{($n$)}}$. 
 Recall $(a^\circ,p^\circ,q^\circ)$ from \eqref{E:GlobalTriple}. Define
$(\bar{a},\bar{p},\bar{q})\in
\mathcal{A}^t\times\mathbb{L}^0(\F^t,\R^d)\times\mathbb{L}^0(\F^t,\R^{d\times d})$ by
\begin{align}\label{E:Step2Limit}
(\bar{a},\bar{p},\bar{q})(s,\omega):=
\begin{cases}
\lim_{n\to \infty} 
(\tilde{a}^{\text{($n$)}},\tilde{p}^{\text{($n$)}},\tilde{q}^{\text{($n$)}})(s,\omega)&\text{if this limit exists,}\\
(a^\circ,p^\circ,q^\circ) &\text{otherwise,}
\end{cases}
\end{align}
where
\begin{align*}
(\tilde{a}^{\text{($n$)}},\tilde{p}^{\text{($n$)}},\tilde{q}^{\text{($n$)}})(s,\omega)&:=
\begin{cases}
(a^{\text{($n$)}},{p}^{\text{($n$)}},{q}^{\text{($n$)}})(s,\omega)&\text{if $s\in
 [0,\tau^{\text{($n$)}}(\omega)]$,}\\
(a^\circ,p^\circ,q^\circ) &\text{otherwise}.
\end{cases}
\end{align*}
Note that, outside of a $\Prob$-evanescent set, the limit in \eqref{E:Step2Limit} always exists, as it is the limit of an eventually constant sequence 
(regarding $\mathrm{Prog}(\F^t)$--$\mathcal{B}(A)\otimes\mathcal{B}(\R^d)\otimes
\mathcal{B}(\R^{d\times d})$ measurability, see, e.g., Lemma~4.29 in \cite{Hitchhiker}).
Further note that \eqref{E:tptq:eps} holds for $(\bar{\tau},\bar{p},\bar{q})$.
Next, we show that $(X^{\text{($n$)}}_{
\tau^{\text{($n$)}}
})_{n\ge 1}$ is a Cauchy sequence in
$\mathbb{L}^2(\cF_{\bar{\tau}},\Prob{\color{black}, \R^d})$. To this end, let $m\ge n\ge 1$. Then
\begin{align*}
&\Mean\left[\abs{
X^{\text{($m$)}}_{
\tau^{\text{($m$)}}
}-X^{\text{($n$)}}_{
\tau^{\text{($n$)}}
}
}^2\right]
\le C \Mean\left[\left(
\int_{\tau^{\text{($n$)}}}^{\tau^{\text{($m$)}}} \abs{b\left(s,X^{(m)}_{\cdot\wedge\varrho^{(m)}_s},\bar{a}_s
\right)+\bar{p}_s}\,ds\right)^2 
\right] \\&\qquad+
C \Mean\left[
\int_{\tau^{\text{($n$)}}}^{\tau^{\text{($m$)}}} \abs{
\sigma\left(s,X^{(m)}_{\cdot\wedge\varrho^{(m)}_s},\bar{a}_s
\right)}^2+\abs{\bar{q}_s}^2\,ds
\right]
\end{align*}
for some constant $C>0$ independent from $m$ and $n$. Noting (A2),
\eqref{E:tptq:eps}  for $(\bar{\tau},\bar{p},\bar{q})$, and $\tau^{\text{($n$)}}\uparrow\bar{\tau}$, we can
apply  the dominated convergence
theorem 
to deduce that the sequence $(X^{\text{($n$)}}_{
\tau^{\text{($n$)}}
})_{n\ge 1}$ is a Cauchy sequence in
$\mathbb{L}^2(\cF_{\bar{\tau}},\Prob,\R^d)$ and we denote its limit 
by $\xi$.
Now, define $\bar{X}:\R_+\times\Omega\to\R^d$ by
\begin{align*}
\bar{X}(s,\omega):=\sum_{n=1}^\infty 
\bfone_{[\tau^{(n-1)}(\omega),\tau^{(n)}(\omega))}(s).X^{(n)}(s,\omega)+\bfone_{[\bar{\tau}(\omega),\infty)}(s).\xi(\omega),
\end{align*}
where $\tau^{(0)}=0$.
We have $\bar{X}\in\mathbb{S}^2(\F,\Prob,\R^d)$ because $(X^{\text{($n$)}}_{
\tau^{\text{($n$)}}
})_{n\ge 1}$
has a subsequence  $(X^{\text{($n_k$)}}_{
\tau^{\text{($n_k$)}}
})_{k\ge 1}$ that converges to $\xi$, $\Prob$-a.s. This also yields that, $\Prob$-a.s., for every $s\ge t$, 
\begin{align*}
\bar{X}_{s\wedge\bar{\tau}}=\mathbf{x}_t+\int_t^{s\wedge\bar{\tau}} 
[b(r,\bar{X}_{\cdot\wedge\bar{\varrho}_r},\bar{a}_r)+\bar{p}_r]
\,dr+
\int_t^{s\wedge\bar{\tau}} 
[\sigma(r,\bar{X}_{\cdot\wedge\bar{\varrho}_r},\bar{a}_r)+\bar{q}_r]
dW_r,
\end{align*}
where 
$\bar{\varrho}:\R_+\times\Omega\to\R_+$ is defined by 
\begin{align*}
\bar{\varrho}(s,\omega):=
\sum_{n=1}^\infty 
\varrho^{(n)}(s,\omega).\bfone_{[\tau^{(n-1)}(\omega),\tau^{(n)}(\omega))}(s)
+s.\bfone_{[\bar{\tau}(\omega),\infty)}(s).
\end{align*}
Finally, 
{\color{black}  by \eqref{E:structural}} and since $v$ is l.s.c.~and bounded from below, we have 
\begin{align*}
\Mean\left[\widehat{v}(\bar{\varrho}_s\wedge\bar{\tau},\bar{X},y)\right]&\le\varliminf_{n\to\infty} 
\Mean\left[ \widehat{v}(\bar{\varrho}_s\wedge\tau^{\text{($n$)}},\bar{X},y)\right]\\
&\le \varliminf_{n\to\infty} \eps\cdot\Mean[s\wedge \tau^{(n)}-t]=
\eps\cdot\Mean[s\wedge\bar{\tau}-t]\quad\text{for all $s\in [t,T]$.}
\end{align*}
We can conclude that 
$(\bar{\varrho},\bar{\tau},\bar{a},\bar{p},\bar{q},\bar{X})\in\mathcal{S}$
and it is an upper bound
of $(\mathfrak{s}^{\text{($n$)}})_{n\ge 1}$.

Thus, 
by the Brezis-Browder ordering principle, $\mathcal{S}$
has an $\mathcal{N}$-maximal element.

\textit{Step~3 }(existence for $\tau= T$, the ``extension step"). 
Fix an $\mathcal{N}$-maximal element
 $\mathfrak{s}^\eps=
 (\tau^\eps,\varrho^\eps,a^\eps,p^\eps,q^\eps,X^\eps)$  of
$\mathcal{S}$ from Step~2. Assume 
 $\Prob(\tau^\eps<T)>0$. 
By the extension lemma (Lemma~\ref{L:Step3}), 
 there is 
an 
$\mathfrak{s}^{\eps,+}=
 (\tau^{\eps,+},\varrho^{\eps,+},a^{\eps,+},p^{\eps,+},q^{\eps,+},X^{\eps,+})
\in\mathcal{S}$ such that $\mathfrak{s}^\eps \preceq\mathfrak{s}^{\eps,+}$ and 
$\Prob(\tau^\eps<\tau^{\eps,+})>0$.  Hence $\mathcal{N}(\mathfrak{s}^\eps)<
\mathcal{N}(\mathfrak{s}^{\eps,+})$, i.e., we have a contradiction to 
$\mathfrak{s}^\eps$ being $\mathcal{N}$-maximal. 
Thus, $\tau^\eps=T$, $\Prob$-a.s.  
\end{proof}

\begin{proof}[Proof of Theorem~\ref{T:Sufficient}]
Fix $\eps>0$, $(t,\mathbf{x},y)\in {\color{black} [0,T]}
\times\Omega\times\R$ 
 with 
 $v(t,\mathbf{x})\le y$.
 By Proposition~\ref{P:epsApproxSol},
 there exists an $\eps$-approximate solution 
 $(\tau,\varrho,a,p,q,X)$  of  \eqref{E:mv:CSDE} for $(\widehat{v},K)$  starting at $(t,\mathbf{x},y)$
 such that $\tau= 
 T$,
 $\Prob$-a.s.  Next, let  $T_1\in (t,T]$. Thanks to (A2), applying   standard estimates (see, e.g., Section~3.2 in \cite{Zhang17book})
 yields
 \begin{align}\label{E1:Proof:Sufficiency}
 \Mean [\norm{X^{t,\mathbf{x},a}_{\cdot\wedge (t+\delta)}-X_{\cdot\wedge (t+\delta)}}_\infty^2]
 \le C_0\,\eps+\int_t^{t+\delta} C_0\Mean [\norm{X^{t,\mathbf{x},a}_{\cdot\wedge s}-X_{\cdot\wedge 
 \varrho_s}}_\infty^2]\,ds
 \end{align}
 for all $\delta\in [0,T-t]$ and some constant $C_0>0$ independent of $\delta$, $\eps$, and $T_1$.
Further note that, by condition~(ii) of Definition~\ref{D:Approx} and by (A2), for every $s\in [0,T_1]$,
\begin{equation}\label{E2:Proof:Sufficiency}
\begin{split}
&\Mean\left[\norm{X_{\cdot\wedge s}-X_{\cdot\wedge\varrho_s}}^2_\infty\right]\le C_1\,\eps \\ &
+ C_1
\Mean\left[
\int_{(s-\eps)\vee 0}^s \abs{b(r,X_{\cdot\wedge\varrho_r},a_r)}^2\
+  \abs{\sigma(r,X_{\cdot\wedge\varrho_r},a_r)}^2\,dr
\right]
\le C_1\,\eps(1+ 2L_b^2)
\end{split}
\end{equation}
for some $C_1>0$ independent of $s$,  $\eps$, and $T_1$.
 Thus we can apply Gronwall's inequality to \eqref{E1:Proof:Sufficiency} combined with \eqref{E2:Proof:Sufficiency}
  to  deduce 
 $ \Mean [\norm{X^{t,\mathbf{x},a}_{\cdot\wedge T_1}-X_{\cdot\wedge T_1}}_\infty^2]\le C\eps$
 for some $C>0$ independent of $\eps$ and $T_1$. Finally,  by (H) and \eqref{E2:Proof:Sufficiency},
 there is a constant $\tilde{C}>0$ independent of $\eps$ and $T_1$ (cf.~\eqref{E:Necc:Proof}) such that
 {\color{black}
\begin{align*}
& \Mean [v(T_1,X^{t,\mathbf{x},a})]\le
 L (C\eps)^{1/2}+\Mean[v(T_1,X)]\\
&\le
 L (C\eps)^{1/2}+ L(C_1\,\eps(1+2L_b^2))^{1/2}+\Mean[v(T_1,X_{\cdot\wedge\varrho_{T_1}})]\\
&\le
L (C\eps)^{1/2}+ L(C_1\,\eps(1+2L_b^2))^{1/2}+L(1+\tilde{C})\eps^{1/2}+\Mean[v(\varrho_{T_1},X)]\\
&\le {\color{black} y+}L (C\eps)^{1/2}+ L(C_1\,\eps(1+2L_b^2))^{1/2}+L(1+\tilde{C})\eps^{1/2}+\eps(T-t). 
\end{align*}}
The last inequality is due to 
Definition~\ref{D:Approx}~(vi).
We can 
see  that \eqref{E:WeakApproxViabDef}  holds.
\end{proof}

\section{Quasi-contingent solutions of HJB equations}\label{S:HJB}
 In this section and the next,  we consider the terminal value problem  
\begin{equation}\label{E:mv_super:HJB}
\begin{split}
\partial_t u+\inf_{a\in A} \left[b(t,\mathbf{x},a)\cdot\partial_{\mathbf{x}} u+
\frac{1}{2} (\sigma\sigma^\top)(t,\mathbf{x},a):\partial^2_{\mathbf{xx}} u \right]&=0\qquad\text{on $[0,T)\times\Omega$,}\\
u(T,\mathbf{x})&=h(\mathbf{x})\qquad\text{on $\Omega$,}
\end{split}
\end{equation}
and the value function  $V^S:[0,T]\times\Omega\to\R$ 
  defined by\footnote{
Recall that $X^{t,\mathbf{x},a}$ solves \eqref{E:mv:CSDE}.
}
\begin{align*}
V^S(t,\mathbf{x}):=\inf_{a\in\mathcal{A}} \Mean\left[h(X^{t,\mathbf{x},a})\right].
\end{align*}
together with $\widehat{V}^S:[0,T]\times\Omega\times\R\to\R$
defined by $\widehat{V}^S(t,\mathbf{x},y):=V^S(t,\mathbf{x})-y$.

\subsection{Quasi-contingent supersolutions}\label{S:Supersol}
For our notion of supersolution, we use the following directional derivative
(cf.~\cite{SubbotinaEtAl85,Subbotina06, Haussmann92MOR, Haussmann94SICON, BK22SICON}
for similar stochastic derivatives and \cite{Carja14SICON} for a related first-order derivative with function-valued direction).

\begin{definition}
Let the function $v$ be bounded from below. 
The \emph{con\-tin\-gent epi\-deri\-va\-tive  $D^{1,2}_\uparrow v(t,\mathbf{x})\,(\mathcal{E})$
of  $v$ at 
$(t,\mathbf{x})\in [0,T)\times\Omega$ in a
multi-valued direction $\mathcal{E}\subset
\mathbb{L}^2(\F^t,\Prob,\R^d)\times\mathbb{L}^2(\F^t,\Prob,\R^{d\times d})$} is defined by
\begin{align*}
&D^{1,2}_\uparrow v(t,\mathbf{x})\,(\mathcal{E})
:=
\sup_{\eps>0}\,\inf\Biggl\{
\Mean\left[\frac{v(t+\delta,X^{t,\mathbf{x};b,\sigma;p,q})-v(t,\mathbf{x})}{\delta}\right]:\,
(b,\sigma)\in\mathcal{E},
\\
&\qquad\qquad
\delta\in (0,\eps\wedge (T-t)],\text{ and 
$(p,q)\in{\color{black}\mathbb{L}^0(\F^t,\R^d)\times\mathbb{L}^0(\F^t,\R^{d\times d})}$ with}\\
&
\qquad\qquad
\norm{\bfone_{\llb t,t+\delta\rrb}.p}^2_{\mathbb{L}^2(\F^t,\Prob,\R^d)}+
\norm{\bfone_{\llb t,t+\delta\rrb}.q}^2_{\mathbb{L}^2(\F^t,\Prob,\R^{d\times d})}\le\eps\delta
\Biggr\},
\end{align*}
where 
$X^{t,\mathbf{x};b,\sigma;p,q}$ satisfies \eqref{E:D:QTS:X}.
\end{definition}

{\color{black} \begin{remark}\label{R:Epi:LimInf}
It is perhaps more intuitive to write $D^{1,2}_\uparrow$ as a limit inferior:
\begin{align*}
D^{1,2}_\uparrow v(t,\mathbf{x})\,(\mathcal{E})=\varliminf_{\substack{\delta\downarrow 0,\\ 
(p,q)\to (0,0)}}
\inf_{(b,\sigma)\in\mathcal{E}} \Mean\left[\frac{v(t+\delta,X^{t,\mathbf{x};b,\sigma;p,q})-v(t,\mathbf{x})}{\delta}\right].
\end{align*}
Note that
analogous directional derivatives with the same order of $\varliminf$ and $\inf$  
already show up in \cite{LionsSouganidis85SICON,SubbotinaEtAl85}. 
Connections of   $D^{1,2}_\uparrow$ with path derivatives of smooth functions can be found in Theorem~\ref{T:Classical:Contingent}
and its proof.
\end{remark}}


The next 
results  link 
contingent epiderivatives with mean quasi-tangent sets from Definition~\ref{D:QTS}. 
This allows the application of the mean viability theorems to  \eqref{E:mv_super:HJB}. 

\begin{lemma}\label{L:D12:EinQTS}
{\color{black} $D^{1,2}_\uparrow v(t,\mathbf{x})\,(\mathcal{E})\le 0\Longleftrightarrow
{\mathcal{E}}\in\mathcal{QTS}_{\widehat{v},K}(t,\mathbf{x},v(t,\mathbf{x}))$.}
\end{lemma}
\begin{proof}
(i) 
If $D^{1,2}_\uparrow v(t,\mathbf{x})\,(\mathcal{E})\le 0$, 
then, for all $n\in\N$, there is a 
$(\delta_n,(b,\sigma)_n,p_n,q_n)\in
(0,\tfrac{1}{n}]\times\mathcal{E}\times
{\color{black}
\mathbb{L}^0(\F^t,\R^d)\times\mathbb{L}^0(\F^t,\R^{d\times d})
}$
such that
\begin{equation}\label{E:quasiContDeriv:quasiContSet} 
\begin{split}
&
\Mean\left[v(t+\delta_n,X^{t,\mathbf{x};(b,\sigma)_n;p_n,q_n})-v(t,\mathbf{x})\right]\delta_n^{-1}
\le n^{-1}
\quad\text{and}\\
&\qquad
 \norm{\bfone_{\llb t,t+\delta_n\rrb}.p_n}^2_{\mathbb{L}^2(\F^t,\Prob,\R^d)}+
\norm{\bfone_{\llb t,t+\delta_n\rrb}.q_n}^2_{\mathbb{L}^2(\F^t,\Prob,\R^{d\times d})}\le \delta_n/n,
\end{split}
\end{equation}
 i.e., $\mathcal{E}\in\mathcal{QTS}_{\widehat{v},(-\infty,0]}(t,\mathbf{x},v(t,\mathbf{x}))$.

(ii) Suppose that 
 $\mathcal{E}\in\mathcal{QTS}_{\widehat{v},(-\infty,0]}(t,\mathbf{x},v(t,\mathbf{x}))$, i.e.,
there exists 
a sequence
$(\delta_n,(b,\sigma)_n,p_n,q_n)_n$ in
$(0,T-t]\times\mathcal{E}\times{\color{black}
\mathbb{L}^0(\F^t,\R^{d})\times\mathbb{L}^0(\F^t,\R^{d\times d})}$
with $\delta_n\le n^{-1}$ such that \eqref{E:quasiContDeriv:quasiContSet} holds. Then
$D^{1,2}_\uparrow v(t,\mathbf{x})\,(\mathcal{E})\le 0$ follows. 
\end{proof}

\begin{lemma}\label{L:QTSinQTS}
{\color{black}
$v(t,\mathbf{x})\le y\Longrightarrow\mathcal{QTS}_{\widehat{v},K}(t,\mathbf{x},v(t,\mathbf{x}))\subset
\mathcal{QTS}_{\widehat{v},K}(t,\mathbf{x},y)$.}
\end{lemma}

\begin{proof}
Let $\mathcal{E}\in \mathcal{QTS}_{\widehat{v},K}(t,\mathbf{x}, v(t,\mathbf{x}))$. 
With $\eps$, $\delta$, and $X$ as in Definition~\ref{D:QTS},
\begin{align*}
\eps\delta\ge \Mean\left[\widehat{v}(t+\delta,X,v(t,\mathbf{x}))\right]=\Mean\left[
v(t+\delta,X)-v(t,\mathbf{x})
\right]\ge \Mean\left[v(t+\delta,X)-y)\right]
\end{align*}
because $v(t,\mathbf{x})\le y$. This
  immediately yields $\mathcal{E}\in \mathcal{QTS}_{\widehat{v},K}(t,\mathbf{x},y)$.
\end{proof}

 Recall  $\mathcal{E}_+$ defined by \eqref{E:E+}.
 
 \begin{definition}
We say $v$
is a \emph{quasi-contingent supersolution} of \eqref{E:mv_super:HJB} if $v$ is l.s.c.~and bounded from below, $v(T,\cdot)\ge h$, and, for each $(t,\mathbf{x})\in [0,T)\times\Omega$,
\begin{align*}
D^{1,2}_\uparrow v(t,\mathbf{x})\,(\mathcal{E}_+(t,\mathbf{x}))\le 0.
\end{align*}
\end{definition}

\begin{theorem}\label{T:VS:super}
If $V^S$ is l.s.c., then it is a quasi-contingent supersolution of \eqref{E:mv_super:HJB}.
\end{theorem}

\begin{proof}
 Fix $\eps>0$ and $(t,\mathbf{x},y)\in [0,T)\times\Omega$ with $y\ge V^S(t,\mathbf{x})$.  
Also let  $s\in [t,T]$. 
 By the dynamic programming principle
(e.g., Theorem~3.5 in \cite{ElKarouiTanII}), 
there is an $a\in\mathcal{A}$  with 
$\Mean[V^S(s,X^{t,\mathbf{x},a})]-V^S(t,\mathbf{x})\le \eps$.
Moreover, by Proposition~4 in \cite{ClaisseEtAl16} (cf.~also Lemma~2.7 in \cite{cosso21v2path}),
 we can assume that $a\in\mathcal{A}^t$. 
Thus (cf.~\eqref{E:WeakApproxViabDef} and Remark~\ref{R:Sufficient:Weaker})
\begin{align*}
{\color{black} \inf_{\tilde{a}\in\mathcal{A}^t} \Mean\left[V^S(s,X^{t,\mathbf{x},\tilde{a}})
\right]\le y.}
\end{align*}
By the proof of Theorem~\ref{T:Necc}, 
$\mathcal{E}_+(t,\mathbf{x})\in  \mathcal{QTS}_{\widehat{V}^S,K}(t,\mathbf{x},y)$.
Hence, by Lemma~\ref{L:D12:EinQTS}, 
$D^{1,2}_\uparrow V^S(t,\mathbf{x})\,(\mathcal{E}_+(t,\mathbf{x}))\le 0$.
\end{proof}

\begin{theorem}\label{T:Super:Com}
Let $v=v(t,\mathbf{x})$ be a  quasi-contingent supersolution of \eqref{E:mv_super:HJB}
that is Lipschitz in $\mathbf{x}$ uniformly in $t$.
 Then  $V^S\le v$.
\end{theorem}

\begin{proof}
Note that, by Lemmata~\ref{L:D12:EinQTS} and \ref{L:QTSinQTS}, 
the following holds:
\begin{equation}\label{E1:Comp1}
\begin{split}
&\forall (t,\mathbf{x},y)\in [0,T)\times\Omega\times\R\,\,
\text{with 
${\color{black} v(t,\mathbf{x})\le y}$}:\\
&\qquad
\mathcal{E}_+(t,\mathbf{x})
\in
\mathcal{QTS}_{\widehat{v},K}(t,\mathbf{x},v(t,\mathbf{x}))\subset 
\mathcal{QTS}_{\widehat{v},K}(t,\mathbf{x},y).
\end{split}
\end{equation}
Now, fix $(t,\mathbf{x})\in [0,T)\times\Omega$.
For every
 $\eps>0$, there exists, by \eqref{E1:Comp1} and Remark~\ref{R:Sufficient:Weaker},
a control 
$a\in\mathcal{A}^t$ such that
${\color{black} \Mean [v(T,X^{t,\mathbf{x},a})]\le v(t,\mathbf{x})+\eps}$.
Thus 
\begin{align*}
V^S(t,\mathbf{x})\le \Mean[h(X^{t,\mathbf{x},a})]\le \Mean [v(T,X^{t,\mathbf{x},a})]\le v(t,\mathbf{x})+\eps.
\end{align*}
This concludes the proof.
\end{proof}

\subsection{Quasi-contingent subsolutions} 
\begin{definition}\label{D:hypo}
Let $v$ be bounded from above.
Then the \emph{contingent hypoderivative  $D^{1,2}_\downarrow v(t,\mathbf{x})\,(1,b,\sigma)$
of $v$ at 
$(t,\mathbf{x})\in [0,T)\times\Omega$ in 
a 
 direction $(1,b,\sigma)\in\R\times
\mathbb{L}^2(\F^t,\Prob,\R^d)\times\mathbb{L}^2(\F^t,\Prob,\R^{d\times d})$} is defined by
\begin{equation}\label{E:D:hypo}
\begin{split}
D^{1,2}_\downarrow v(t,\mathbf{x})\,(1,b,\sigma)&:=
\inf_{\eps>0}\,\sup\Biggl\{
\Mean\left[\frac{v(t+\delta,X^{t,\mathbf{x};b,\sigma;p,q})-v(t,\mathbf{x})}{\delta}\right]:\,
\\
&\quad\delta\in (0,\eps],\text{ and 
$(p,q)\in{\color{black}\mathbb{L}^0(\F^t,\R^d)\times\mathbb{L}^0(\F^t,\R^{d\times d})}$ with}\\
&\quad\norm{\bfone_{\llb t,t+\delta\rrb}.p}^2_{\mathbb{L}^2(\F^t,\Prob,\R^d)}+
\norm{\bfone_{\llb t,t+\delta\rrb}.q}^2_{\mathbb{L}^2(\F^t,\Prob,\R^{d\times d})}\le\eps\delta
\Biggr\},
\end{split}
\end{equation}
where 
$X^{t,\mathbf{x};b,\sigma;p,q}$ satisfies \eqref{E:D:QTS:X}.
\end{definition}

{\color{black}
\begin{remark}
We may write 
$D^{1,2}_\downarrow$ as a limit superior (cf.~Remark~\ref{R:Epi:LimInf}):
\begin{align*}
D^{1,2}_\downarrow v(t,\mathbf{x})\,(1,b,\sigma)=\varlimsup_{\substack{\delta\downarrow 0,\\ 
(p,q)\to (0,0)}}
 \Mean\left[\frac{v(t+\delta,X^{t,\mathbf{x};b,\sigma;p,q})-v(t,\mathbf{x})}{\delta}\right].
\end{align*}
\end{remark}
}

 \begin{definition}
We call $v$  a  \emph{quasi-contingent subsolution} of \eqref{E:mv_super:HJB} if $v$ is u.s.c.~and bounded from above, $v(T,\cdot)\le h$, and
 for all 
 $(t,\mathbf{x})\in [0,T)\times\Omega$,
\begin{align*}
\inf_{a\in\mathcal{A}^{\color{black}t}} D^{1,2}_\downarrow v(t,\mathbf{x})\,(1,(b,\sigma)(\cdot,\mathbf{x}_{\cdot\wedge t},
a))\ge 0.  
\end{align*}
\end{definition}

Given $a\in\mathcal{A}$ and $(t,\mathring{\mathbf{x}})\in\R_+\times\Omega$, define $a^{t,\mathring{\mathbf{x}}}\in\mathcal{A}^{\color{black} t}$ by
 \begin{align*}
 a^{t,\mathring{\mathbf{x}}}(s,\omega):=a(s,
 \bfone_{[0,t)}.\mathring{\mathbf{x}}+\bfone_{[t,\infty)}.(\mathring{\mathbf{x}}_t+\omega-\omega_t)).
 \end{align*}
 
 For fixed  $a\in\mathcal{A}$, we consider here, instead of \eqref{E:mv:CSDE},
  the systems
 \begin{equation}\label{E0:mv:CSDE:sub}
\begin{split}
&dX_s^{t,\mathbf{x},\mathring{\mathbf{x}},a}=b(s,X^{t,\mathbf{x},\mathring{\mathbf{x}},a},a^{t,\mathring{\mathbf{x}}}_s)\,ds+
\sigma(s,X^{t,\mathbf{x},\mathring{\mathbf{x}},a},a^{t,\mathring{\mathbf{x}}}_s)\,dW_s
\,\,\,\text{on $(t,\infty)$, $\Prob$-a.s.,}\\
&X^{t,\mathbf{x},\mathring{\mathbf{x}},a}_s=\mathbf{x}_s\quad\text{on $[0,t]$,}
\end{split}
\end{equation}
and, instead of \eqref{E:E+}, 
$\mathcal{E}^a_+(t,\mathbf{x},\mathring{\mathbf{x}}):=
\{\bfone_{[0,T]}.(b,\sigma)(\cdot,\mathbf{x}_{\cdot\wedge t},a^{t,\mathring{\mathbf{x}}})\}$,
where $(t,\mathbf{x},\mathring{\mathbf{x}})\in\R_+\times\Omega\times\Omega$.

\begin{definition}
Fix $a\in\mathcal{A}$. The pair $(\widehat{v},K)$ is \emph{mean viable} 
for
\eqref{E0:mv:CSDE:sub} if, for  all  $s\in [t,T]$ and 
$(t,\mathbf{x},\mathring{\mathbf{x}},y)\in\R_+\times\Omega\times\Omega\times\R$ 
 with ${\color{black}v(t,\mathbf{x})\le y}$,
 ${\color{black} \Mean\left[v(s,X^{t,\mathbf{x},\mathring{\mathbf{x}},a}\right]\le y}$. 
\end{definition}

The next remark contains counterparts to Theorems~\ref{T:Necc} and ~\ref{T:Sufficient}, whose proofs have
no essential differences and are omitted.

\begin{remark}\label{R:Necc:Suff}
 Let $u=u(t,\mathbf{x}): [0,T]\times\Omega\to\R$ be l.s.c.,
Lipschitz in $\mathbf{x}$ uniformly in $t$, and bounded from below.
Define $\widehat{u}:[0,T]\times\Omega\times\R\to\R$ by
 $\widehat{u}(t,\mathbf{x},y):=u(t,\mathbf{x})-y$.
Fix $a\in\mathcal{A}$. Then
$(\widehat{u},K)$ is 
mean viable for
\eqref{E0:mv:CSDE:sub} if and only if 
${\color{black} u(t,\mathbf{x})\le y}$
implies 
$\mathcal{E}^a_+(t,\mathbf{x},\mathring{\mathbf{x}})\in
\mathcal{QTS}_{\widehat{u},K}(t,\mathbf{x},y)$
for every $(t,\mathbf{x},\mathring{\mathbf{x}},y)\in [0,T)\times\Omega\times\Omega\times\R$.

As in Theorem~\ref{T:Necc},
we do not need  $u$ being Lipschitz in $\mathbf{x}$ and bounded from below for 
the necessary condition for 
mean viability.
\end{remark}

\begin{theorem}\label{T:VS:sub}
If $V^S$ is u.s.c, then it is a quasi-contingent subsolution of \eqref{E:mv_super:HJB}.
\end{theorem}

\begin{proof}
Let 
$U:=-V^S$ and 
$\widehat{U}(t,\mathbf{x},y):=U(t,\mathbf{x})-y$.
Fix $a\in\mathcal{A}$.
By the dynamic programming principle (see, e.g., Theorem~3.5 in \cite{ElKarouiTanII}),
$V^S(t,\mathbf{x})\le \Mean[V^S(s,X^{t,\mathbf{x},\mathring{\mathbf{x}},a})]$. Thus 
${\color{black}\Mean[U(s,X^{t,\mathbf{x},\mathring{\mathbf{x}},a}]\le y}$ for every
$(t,\mathbf{x},\mathring{\mathbf{x}})\in[0,T]\times\Omega\times\Omega$, 
$y\ge U(t,\mathbf{x})$, and 
$s\in [t,T]$,
i.e.,
$(\widehat{U},K)$ is mean viable for \eqref{E0:mv:CSDE:sub}.
Hence, by Re\-mark~\ref{R:Necc:Suff}, we have  
$\mathcal{E}^a_+(t,\mathbf{x},\mathring{\mathbf{x}})\in
\mathcal{QTS}_{\widehat{U},K}(t,\mathbf{x},y)$.
Invoking Lemma~\ref{L:D12:EinQTS}  concludes the proof.
\end{proof}

\begin{theorem}\label{T:Comp:Sub}
Let $v=v(t,\mathbf{x})$ be a  quasi-contingent subsolution of \eqref{E:mv_super:HJB}
that is Lipschitz in $\mathbf{x}$ uniformly in $t$. 
Then  $v\le V^S$.
\end{theorem}
\begin{proof}
Let $u:=-v$ and  $\widehat{u}(t,\mathbf{x},y):=u(t,\mathbf{x})-y$.
Fix $a\in\mathcal{A}$. 
Note that
 \begin{align*}
   D^{1,2}_\downarrow v(t,\mathbf{x})\,(1,(b,\sigma)(\cdot,\mathbf{x}_{\cdot\wedge t},a^{t,\mathring{\mathbf{x}}}))\ge 0
 \Longleftrightarrow
 D^{1,2}_\uparrow u(t,\mathbf{x})\,(\mathcal{E}^a_+(t,\mathbf{x},\mathring{\mathbf{x}}))\le 0.
 \end{align*}
Thus, using Remark~\ref{R:Necc:Suff} 
instead of Remark~\ref{R:Sufficient:Weaker}, 
 we can proceed as in the proof of Theorem~\ref{T:Super:Com}  to
 deduce that 
  $\Mean \left[\widehat{u}(T,X^{t,\mathbf{x},\mathring{\mathbf{x}}},u(t,\mathbf{x}))\right]\le 0.$ Hence,  
 \begin{align*}
 \Mean\left[h(X^{t,\mathbf{x},{\color{black}\mathring{\mathbf{x}},}a})\right]\ge \Mean \left[v(T,X^{t,\mathbf{x},\mathring{\mathbf{x}},a})\right]\ge v(t,\mathbf{x}).
 \end{align*}
  {\color{black} Finally, thanks to the proof of  Proposition~4~(i) in \cite{ClaisseEtAl16}, we obtain}
  \begin{align*}
  V^S(t,\mathbf{x}){\color{black}=\inf_{a\in\mathcal{A}} 
  \int_{\Omega} \Mean\left[h(X^{t,\mathbf{x},\mathring{\mathbf{x}},a})\right]\,\Prob(d\mathring{\mathbf{x}})}\ge v(t,\mathbf{x}).
  \end{align*}
This concludes the proof.
\end{proof}

\subsection{Comparison}
The next result follows  directly from
 Theorems~\ref{T:Super:Com} and \ref{T:Comp:Sub}.
\begin{theorem}\label{T:comparison}
Let $v_-=v_-(t,\mathbf{x})$ be a quasi-contingent sub- and $v_+=v_+(t,\mathbf{x})$ be a
quasi-contingent supersolution of  \eqref{E:mv_super:HJB}. 
Suppose that  $v_-$ and $v_+$ are Lipschitz in $\mathbf{x}$ uniformly in $t$.
Then $v_-\le v_+$.
\end{theorem}

{\color{black}
\subsection{Quasi-contingent solutions}

\begin{definition}
We call $v$ a \emph{quasi-contingent solution} of   \eqref{E:mv_super:HJB} if
$v$ is both, a quasi-contingent sub- and supersolution of   \eqref{E:mv_super:HJB}.
\end{definition}

\begin{theorem}
Let $h$ be Lipschitz. Then the value function
 $V^S=V^S(t,\mathbf{x})$ is the unique quasi-contingent solution of  \eqref{E:mv_super:HJB} that is Lipschitz in $\mathbf{x}$ uniformly in $t$.
\end{theorem}
\begin{proof}
First note, that $V^S=V^S(t,\mathbf{x})$ is continuous and Lipschitz in $\mathbf{x}$ uniformly in $t$ (see, e.g., Proposition~2.6 in \cite{cosso21v2path}).
Thus, by Theorems~\ref{T:VS:super} and \ref{T:VS:sub}, $V^S$ is a quasi-contingent solution of 
 \eqref{E:mv_super:HJB}. Theorem~\ref{T:comparison} yields uniqueness.
\end{proof}


\section{Classical  and viscosity solutions of HJB equations}\label{S:ClassicalAndViscosity}
For the sake of notational simplicity, 
 let $d=1$ in this section.  All results hold for higher dimensions subject to obvious modifications.


Given $(t,\mathbf{x})\in [0,T)\times\Omega$ and
$(b,{\sigma})\in\mathbb{L}^2(\F^t,\Prob,\R)\times\mathbb{L}^2(\F^t,\Prob,\R)$,
recall $X^{t,\mathbf{x};b,\sigma;p,q}$ from \eqref{E:D:QTS:X} and  
define $X^{t,\mathbf{x};b,\sigma}:=X^{t,\mathbf{x};b,\sigma;0,0}$, i.e.,
\begin{equation}\label{E:Xbsigma}
\begin{split}
dX^{t,\mathbf{x};b,\sigma}_s&=b_s\,ds+\sigma_s\,dW_s\quad\text{on $(t,T]$, $\Prob$-a.s.,}\\
X^{t,\mathbf{x};b,\sigma}_s&=\mathbf{x}_s\quad\text{on $[0,t]$}.
\end{split}
\end{equation}

\subsection{Path derivatives and path differential operators}
A functional It\^o calculus with first and second order path derivatives was introduced in  \cite{dupirefunctional}.
In contrast to  \cite{dupirefunctional},  we use an implicit definition for our path derivatives (adapted from \cite{ETZ_I}).
\begin{definition}
Let $t\in [0,T)$. The space $C^{1,2}_b([t,T]\times\Omega)$ consists of all bounded and continuous
functions $\varphi:[t,T]\times\Omega\to\R$ for which there are bounded and continuous
functions $\partial_t \varphi$, $\partial_{\mathbf{x}} \varphi$, $\partial^2_{\mathbf{x}\mathbf{x}}\varphi:[t,T]\times\Omega\to\R$ 
such that, for all $\mathbf{x}\in\Omega$, all $(b,{\sigma})\in\mathbb{L}^2(\F^t,\Prob,\R)\times\mathbb{L}^2(\F^t,\Prob,\R)$ and all $t_1$, $t_2\in [t,T]$ with $t_1<t_2$,
\begin{equation}\label{E:functionalIto}
\begin{split}
\varphi(t_2,X)-\varphi(t_1,X)&=\int_{t_1}^{t_2} \partial_t \varphi(s,X)+b_s\,\partial_{\mathbf{x}}\varphi(s,X)+
\frac{1}{2}\sigma^2_s\,\partial^2_{\mathbf{xx}}\varphi(s,X)\,ds\\ &\qquad+
\int_{t_1}^{t_2}\sigma_s\,\partial_{\mathbf{x}}\varphi(s,X)\,dW_s,\text{ $\Prob$-a.s.,}
\end{split}
\end{equation}
where $X=X^{t,\mathbf{x};b,\sigma}$  (see \eqref{E:Xbsigma}).
\end{definition}

\begin{remark}
The \emph{path derivatives} $\partial_t \varphi$,  $\partial_{\mathbf{x}} \varphi$, and $\partial^2_{\mathbf{x}\mathbf{x}}\varphi$
 of any function $\varphi$ in $C^{1,2}_b([t,T]\times\Omega)$  are uniquely determined (see \cite{ETZ_I } for details).
\end{remark}

\begin{definition}
Let $\varphi\in C^{1,2}_b([0,T]\times\Omega)$ and $(t,\mathbf{x})\in [0,T)\times\Omega$.
 For $a\in A$, define  
\begin{align}\label{E:La1}
\mathcal{L}^a\varphi (t,\mathbf{x}):=\partial_t \varphi(t,\mathbf{x})+
b(t,\mathbf{x},a)\,\partial_{\mathbf{x}} \varphi(t,\mathbf{x})+\frac{1}{2}\sigma^2(t,\mathbf{x},a)\,\partial^2_{\mathbf{xx}}\varphi(t,\mathbf{x}).
\end{align}
For  $a\in \mathcal{A}$, define, with slight abuse of notation,\footnote{
There should be no danger confusion because if $a\in\mathcal{A}$ is constant after time $t$, 
say $a(s,\omega)=\tilde{a}\in A$ for each $(s,\omega)\in [t,T]\times\Omega$, then
$\mathcal{L}^a \varphi(t,\mathbf{x})$ in the sense of \eqref{E:La2} coincides with
$\mathcal{L}^{\tilde{a}}\varphi(t,\mathbf{x})$ in the sense of \eqref{E:La1}. Here, it is assumed
that $b(\cdot,\cdot,\cdot)$ and $\sigma(\cdot,\cdot,\cdot)$ from \eqref{E:Data} are continuous.
}
\begin{align}\label{E:La2}
\mathcal{L}^a\varphi (t,\mathbf{x}):=
\varlimsup_{\substack{\delta\downarrow 0,\\ 
(p,q)\to (0,0)}}
 \frac{1}{\delta}\Mean\left[
\int_t^{t+\delta} \mathcal{L}^{a(s)}\varphi(s,X^{p,q})\,ds
\right],
\end{align}
where 
$dX^{p,q}_s=[b(s,\mathbf{x}_{\cdot\wedge t},a(s))+p_s]\,ds+[\sigma(s,\mathbf{x}_{\cdot\wedge t},a(s))+q_s]\,dW_s$ on $[t,T]$
with $X^{p,q}_{\cdot\wedge t}=\mathbf{x}_{\cdot\wedge t}$ and the limit superior in \eqref{E:La2} is
to be understood in the sense of the right-hand side of \eqref{E:D:hypo}.
\end{definition}

\subsection{Classical solutions}
\begin{definition}
 A function $u\in C^{1,2}_b([0,T]\times\Omega)$ 
 is a \emph{classical solution} 
  of \eqref{E:mv_super:HJB} if $u(T,\cdot)=h $  
   and
  $\inf_{a\in A}\mathcal{L}^a v(t,\mathbf{x}) = 0$  
   for all
  $(t,\mathbf{x})\in [0,T)\times\Omega$. 
\end{definition}


\begin{theorem}\label{T:Classical:Contingent}
Let 
the data from \eqref{E:Data} be continuous.  
Then
classical solutions of \eqref{E:mv_super:HJB} are quasi-contingent solutions of \eqref{E:mv_super:HJB}.
\end{theorem}

\begin{proof}
Let $u$ be a classical solution of \eqref{E:mv_super:HJB}.  Fix $(t,\mathbf{x})\in [0,T)\times\Omega$. 
First, we establish the quasi-contingent supersolution property. 
To this end, we use 
\begin{equation}\label{E:tE+}
\begin{split}
&\mathcal{E}^\circ_+(t,\mathbf{x}):=
\{({b},{\sigma})\in\mathbb{L}^2(\F^t,\Prob,\R)\times\mathbb{L}^2(\F^t,\Prob,\R):\,\exists a\in A:\\
&\qquad
 (b,\sigma)(s,\omega)=(b,\sigma)(s,\mathbf{x}_{\cdot\wedge t},a)
 \quad\text{$dt\times d\Prob$-a.e.~on $[t,T]\times\Omega$}\}. 
 \end{split}
 \end{equation}
 Recall \eqref{E:Xbsigma}. Then, by \eqref{E:functionalIto},
\begin{align*}
D^{1,2}_\uparrow u(t,\mathbf{x})\,(\mathcal{E}_+(t,\mathbf{x})) &\le
\sup_{\eps>0} \inf_{\substack{(b,\sigma)\in\mathcal{E}^\circ_+(t,\mathbf{x}),\\\delta\in (0,\eps\wedge(T-t))}}
\Mean\left[
\frac{u(t+\delta,X^{t,\mathbf{x};b,\sigma})-u(t,\mathbf{x})}{\delta}
\right]\\
&\le \inf_{(b,\sigma)\in\mathcal{E}^\circ_+(t,\mathbf{x})}\varliminf_{\delta\downarrow 0}
\Mean\left[
\frac{u(t+\delta,X^{t,\mathbf{x};b,\sigma})-u(t,\mathbf{x})}{\delta}
\right]\\
&= \inf_{a\in A} \mathcal{L}^a u(t,\mathbf{x})=0,
\end{align*}
i.e., $u$ is a quasi-contingent supersolution of   \eqref{E:mv_super:HJB}.
Next, we establish the quasi-contingent subsolution property.
 Note that, for each $a\in\mathcal{A}^t$, 
 \eqref{E:functionalIto} yields
\begin{align*}
 D^{1,2}_\downarrow u(t,\mathbf{x})\,(1,(b,\sigma)(\cdot,\mathbf{x}_{\cdot\wedge t},a))
& \ge \varlimsup_{\delta\downarrow 0}\Mean\left[\frac{
 u(t+\delta,X^{t,\mathbf{x};(b,\sigma)(\cdot,\mathbf{x}_{\cdot\wedge t},a)})-u(t,\mathbf{x})
 }{\delta} \right]\\
 &=\varlimsup_{\delta\downarrow 0}\Mean\left[\frac{
 \int_t^{t+\delta} \mathcal{L}^{a(s)} u(s,X^{t,\mathbf{x};(b,\sigma)(\cdot,\mathbf{x}_{\cdot\wedge t},a)})\,ds
 }{\delta} \right]\\
 &\ge \varlimsup_{\delta\downarrow 0}\Mean\left[\frac{
 \int_t^{t+\delta} \inf_{\tilde{a}\in A} \mathcal{L}^{\tilde{a}} 
 u(s,X^{t,\mathbf{x};(b,\sigma)(\cdot,\mathbf{x}_{\cdot\wedge t},a)})\,ds
 }{\delta} \right]\\
&= \inf_{\tilde{a}\in A} \mathcal{L}^{\tilde{a}} u(t,\mathbf{x})=0,
\end{align*}
i.e., $u$ is a quasi-contingent subsolution of   \eqref{E:mv_super:HJB}, which concludes the proof.
\end{proof}



\subsection{Viscosity solutions}\label{SS:ViscSol}
The notions we employ here are in the framework of the notions of viscosity solutions in \cite{cosso18,EKTZ11,ETZ_I}
for path-dependent PDEs that use tangency in mean as opposed to  the usually used pointwise tangency.
 


First, we  introduce  test function spaces  for our notion of viscosity solutions.  
They are very similar to the spaces in Remark~6.7 of
\cite{cosso18}.
  For $u:[0,T]\times\Omega\to\R$, $(t,\mathbf{x})\in [0,T)\times\Omega$, and
 $\mathcal{E}\subset
\mathbb{L}^2(\F^t,\Prob,\R)\times\mathbb{L}^2(\F^t,\Prob,\R)$,
 put
\begin{equation*}
\begin{split}
\overline{\Phi}^{\mathcal{E}}u(t,\mathbf{x})&:=
\bigl\{\varphi\in C^{1,2}_b([t,T]\times\Omega):
\exists\eps>0:\forall \delta\in (0,\eps]:\forall (b,\sigma;p,q)\in\mathcal{E}\times\mathbf{O}_{\eps,\delta}^t:\\
&\qquad\qquad
 0=(\varphi-u)(t,\mathbf{x})   \ge\Mean\left[ 
 (\varphi-u)(t+\delta,X^{t,\mathbf{x};b,\sigma;p,q})
 \right]\bigr\},\\
 \underline{\Phi}^{\mathcal{E}}u(t,\mathbf{x})&:=
 \bigl\{\varphi\in C^{1,2}_b([t,T]\times\Omega):
\exists\eps>0:\forall \delta\in (0,\eps]:
\forall (b,\sigma;p,q)\in\mathcal{E}\times\mathbf{O}_{\eps,\delta}^t:\\
&\qquad\qquad
 0=(\varphi-u)(t,\mathbf{x})\le \Mean\left[ 
 (\varphi-u)(t+\delta,X^{t,\mathbf{x};b,\sigma;p,q})
 \right]\bigr\},
\end{split}
\end{equation*}
where $X^{t,\mathbf{x};b,\sigma;p,q}$ has been defined in \eqref{E:D:QTS:X} and 
\begin{align*}
\mathbf{O}_{\eps,\delta}^t:=
\left\{(p,q)\in\mathbb{L}^0(\F^t,\R)^2:\,
\norm{\bfone_{\llb t,t+\delta\rrb}.p}^2_{\mathbb{L}^2(\F^t,\Prob,\R)}+
\norm{\bfone_{\llb t,t+\delta\rrb}.q}^2_{\mathbb{L}^2(\F^t,\Prob,\R)}\le\eps\delta\right\}.
\end{align*}


Next, recall $\mathcal{E}_+(t,\mathbf{x})$ from \eqref{E:E+}
and, given $a\in\mathcal{A}$, put 
$\mathcal{E}_{+}^a(t,\mathbf{x}):=\{(b,\sigma)(\cdot,\mathbf{x}_{\cdot\wedge t},a)\}$.

\begin{definition} Consider a function $u:[0,T]\times\Omega\to\R$.

(i) We call  $u$
 a \emph{viscosity supersolution} of \eqref{E:mv_super:HJB} if $u$ is l.s.c.~and bounded from below, $u(T,\cdot)\ge h$, and,
 for all $(t,\mathbf{x})\in [0,T)\times\Omega$ and 
 $\varphi\in\overline{\Phi}^{\mathcal{E}_+(t,\mathbf{x})}\,u(t,\mathbf{x})$, 
\begin{align*}
\inf_{a\in A} \mathcal{L}^a \varphi(t,\mathbf{x})\le 0.
\end{align*}

(ii) We call  $u$ a \emph{viscosity subsolution} of \eqref{E:mv_super:HJB} if $u$ is u.s.c.~and bounded from above, $u(T,\cdot)\le h$, and,
 for all $(t,\mathbf{x})\in [0,T)\times\Omega$, 
  $a\in\mathcal{A}^t$, and $\varphi\in\underline{\Phi}^{\mathcal{E}_+^a(t,\mathbf{x})}\,u(t,\mathbf{x})$, 
\begin{align*}
 \mathcal{L}^a \varphi(t,\mathbf{x})\ge 0.  
\end{align*}

(iii) We call $u$ a \emph{viscosity solution} of \eqref{E:mv_super:HJB} if $u$ is both,
a viscosity sub- and supersolution of \eqref{E:mv_super:HJB}.
\end{definition}

\begin{remark}
Our definition of viscosity supersolution is nearly identical with the one in  Definition~6.8 in \cite{cosso18}.
The viscosity subsolution case is slightly different.
A function $u$ is a viscosity subsolution of  $\inf_{a\in A} \mathcal{L}^a u(t,\mathbf{x})=0$ in our sense if it 
is a viscosity subsolution of $\mathcal{L}^a  u(t,\mathbf{x})=0$ for every $a\in\mathcal{A}^t$
in the sense of Definition~4.6 in \cite{cosso18} (up to minor technical differences).
\end{remark}


\begin{theorem} 
Let the data from \eqref{E:Data} be continuous.  
Let $u:[0,T]\times\Omega\to\R$ be a
classical solution of \eqref{E:mv_super:HJB}.  Then $u$ is a  viscosity solution of \eqref{E:mv_super:HJB}.
\end{theorem}

\begin{proof}
Fix $(t,\mathbf{x})\in [0,T)\times\Omega$.
First, we establish the viscosity supersolution property. 
Fix  $(b,\sigma)\in \mathcal{E}^\circ_+(t,\mathbf{x})$ (see \eqref{E:tE+}) 
with corresponding 
$a\in A$, recall \eqref{E:Xbsigma}, and 
let $\varphi\in\overline{\Phi}^{\mathcal{E}_+(t,\mathbf{x})}\,u(t,\mathbf{x})$,
i.e., there is an $\eps>0$ such that, for all $\delta\in (0,\eps]$, 
\begin{align*}
\Mean\left[\varphi(t+\delta,X^{t,\mathbf{x};b,\sigma})-\varphi(t,\mathbf{x})\right]\le
\Mean\left[u(t+\delta,X^{t,\mathbf{x};b,\sigma})-u(t,\mathbf{x})\right].
\end{align*}
 Dividing 
this inequality  
 by $\delta$ and letting $\delta\downarrow 0$ yields
$\mathcal{L}^a\varphi(t,\mathbf{x})\le \mathcal{L}^a u(t,\mathbf{x})$, i.e.,
 $\inf_{\tilde{a}\in A}\mathcal{L}^{\tilde{a}}\varphi(t,\mathbf{x})\le 
\inf_{\tilde{a}\in A}\mathcal{L}^{\tilde{a}} u(t,\mathbf{x})=0$.

Next, we establish the viscosity subsolution property. To this end,
fix $a\in\mathcal{A}^t$ and $\varphi\in\underline{\Phi}^{\mathcal{E}_+^a(t,\mathbf{x})}\,u(t,\mathbf{x})$. Then
$(b,\sigma)\in  \mathcal{E}_+^a(t,\mathbf{x})$ implies
\begin{align*}
\Mean\left[\varphi(t+\delta,X^{t,\mathbf{x};b,\sigma})-\varphi(t,\mathbf{x})\right]\ge
\Mean\left[u(t+\delta,X^{t,\mathbf{x};b,\sigma})-u(t,\mathbf{x})\right].
\end{align*}
Hence
$\mathcal{L}^a\varphi (t,\mathbf{x})\ge
\varlimsup_{\delta\downarrow 0} \frac{1}{\delta}\Mean\left[
\int_t^{t+\delta} \mathcal{L}^{a(s)}u(s,X^{t,\mathbf{x};b,\sigma})\,ds
\right]\ge 0$.
\end{proof}

\begin{theorem}\label{T:viscositysuper}
If  $u:[0,T]\times\Omega\to\R$ is a viscosity supersolution of \eqref{E:mv_super:HJB}, then it is a quasi-contingent supersolution of \eqref{E:mv_super:HJB}.
\end{theorem}

\begin{proof}
Assume 
 $u$ is not a quasi-contingent supersolution of  \eqref{E:mv_super:HJB}. Then there is a pair
 $(t,\mathbf{x})\in[0,T)\times\Omega$ such that $D^{1,2}_\uparrow u(t,\mathbf{x})(\mathcal{E}_+(t,\mathbf{x}))>c$
 for some constant $c>0$, i.e., there is an $\eps>0$ such that, for all $\delta\in (0,\eps]$,
  $(b,\sigma)\in\mathcal{E}_+(t,\mathbf{x})$,
 and $(p,q)\in\mathbf{O}_{\eps,\delta}^t$, 
 \begin{align*}
 \Mean[u(t+\delta,X^{t,\mathbf{x};b,\sigma;p,q})-u(t,\mathbf{x})]>c\delta.
 \end{align*}
 Hence the function
 $\varphi:[t,T]\times\Omega\to\R$ defined by $\varphi(s,\omega):=u(t,\mathbf{x})+c\cdot(s-t)$
 belongs to $\overline{\Phi}^{\mathcal{E}_+(t,\mathbf{x})}\,u(t,\mathbf{x})$. However,
 $\inf_{a\in A}\mathcal{L}^a \varphi(t,\mathbf{x})=\partial_t \varphi(t,\mathbf{x})=c>0$, which is a contradiction to $u$ being a viscosity
 supersolution of  \eqref{E:mv_super:HJB}.
\end{proof}

\begin{theorem}\label{T:viscositysub}
If  $u:[0,T]\times\Omega\to\R$ is a viscosity subsolution of \eqref{E:mv_super:HJB}, then it is a quasi-contingent subsolution of \eqref{E:mv_super:HJB}.
\end{theorem}

\begin{proof}
Assume 
 $u$ is not a quasi-contingent subsolution of \eqref{E:mv_super:HJB}.
Then there are 
 $(t,\mathbf{x})\in [0,T)\times\Omega$ and 
 $a\in\mathcal{A}^t$ such that
$D^{1,2}_\downarrow u(t,\mathbf{x})(1,(b,\sigma)(\cdot,\mathbf{x}_{\cdot\wedge t},a))<-c$
for some 
 $c>0$, i.e.,  
 there is an $\eps>0$ such that, for all $\delta\in (0,\eps]$
and $(p,q)\in\mathbf{O}_{\eps,\delta}^t$,
\begin{align*}
\Mean[
 u(t+\delta,X^{t,\mathbf{x};(b,\sigma)(\cdot,\mathbf{x}_{\cdot\wedge t},a);p,q})-u(t,\mathbf{x})
 ]<-c\delta.
\end{align*}
Define $\varphi(s,\omega):=u(t,\mathbf{x})-c\cdot(s-t)$, $(s,\omega)\in [t,T]\times\Omega$.
Then $\varphi\in\underline{\Phi}^{\mathcal{E}_+^a(t,\mathbf{x})}\,u(t,\mathbf{x})$
 but $\mathcal{L}^a \varphi(t,\mathbf{x})=-c<0$.
This is a contradiction to $u$ being a viscosity subsolution.
\end{proof}

\begin{remark}
It would be very interesting to obtain counterparts of Theorems~\ref{T:viscositysuper} and \ref{T:viscositysub}
with the usual ``Crandall--Lions" notion of viscosity solutions 
(as used in
\cite{CossoRusso22,cosso21v2path,zhou2020viscosity} for path-dependent PDEs).
While equivalence results between viscosity solutions and contingent solutions (or similar nonsmooth solutions)
  have been established in the first-order case for standard PDEs on finite-dimensional spaces
 (see, e.g., \cite{BardiCapuzzoDolcetta,Frankowska93SICON,LionsSouganidis85SICON,SubbotinBook}),
 for PDEs on Hilbert space in \cite{ClarkeLedyaev94TAMS},
for PDEs on Wasserstein space in \cite{BadreddineFrankowska22}, and for path-dependent PDEs in 
\cite{GP23JFA}, obtaining corresponding results in the second-order case remains a challenging largely open problem.
\end{remark}

\subsection{Comparison for viscosity solutions}
The following  result 
is a direct consequence of Theorems~\ref{T:comparison},
\ref{T:viscositysuper}, and \ref{T:viscositysub}.

\begin{theorem}\label{T:viscosity:comparison}
Let $u_-=u_-(t,\mathbf{x})$ be a viscosity sub- and $u_+=u_+(t,\mathbf{x})$ be
a viscosity supersolution of \eqref{E:mv_super:HJB}. If
$u_-$ and $u_+$ are Lipschitz in $\mathbf{x}$ uniformly in $t$, then $u_-\le u_+$.
\end{theorem}

\subsection{Existence and uniqueness for viscosity solutions}
\begin{theorem}\label{T:EU:Viscosity} Let 
the data from \eqref{E:Data} be continuous. Moreover, assume that $h$ is Lipschitz and 
$(t,\mathbf{x},a)\mapsto (b,\sigma)(t,\mathbf{x},a)$ is uniformly continuous in $t$
uniformly in $(\mathbf{x},a)$. Then the value function $V^S$ is the unique viscosity solution of \eqref{E:mv_super:HJB}.
\end{theorem}

\begin{proof}
We  sketch parts of  the proof for existence, as the arguments are essentially the same as in the proof of Theorem~3.4
in \cite{cosso21v2path} (however in  \cite{cosso21v2path} a maximization problem is studied in contrast to our
minimization problem).
First  note that 
 $V^S=V^S(t,\mathbf{x})$ is continuous and Lipschitz in $\mathbf{x}$ uniformly in $t$
 by Proposition~2.6 in \cite{cosso21v2path}.

Next, fix $(t,\mathbf{x})\in [0,T)$. We establish now the viscosity supersolution property.
To this end, let  $\varphi\in\overline{\Phi}^{\mathcal{E}_+(t,\mathbf{x})}\,V^S(t,\mathbf{x})$
with corresponding $\eps>0$. Note that, for every $\delta>0$, the dynamic programming principle
provides us with an $a^\delta\in\mathcal{A}^t$ 
(cf.~the proof of Theorem~\ref{T:VS:super}) such that
\begin{align}\label{E:ViscSuperVeric}
V^S(t,\mathbf{x})+\delta^2&\ge \Mean\left[V^S(t+\delta,X^{t,\mathbf{x},a^\delta})\right]
=\Mean\left[V^S(t+\delta,X^{t,\mathbf{x};b^\delta,\sigma^\delta;p^\delta,q^\delta})\right],
\end{align}
where
\begin{align*}
(b^\delta_s,\sigma^\delta_s)&:=(b,\sigma)(s,\mathbf{x}_{\cdot\wedge t},a^\delta_s),\\
(p^\delta_s,q^\delta_s)&:=(b(s,X^{t,\mathbf{x},a^\delta},a^\delta_s)-b^\delta_s,\,
\sigma(s,X^{t,\mathbf{x},a^\delta},a^\delta_s)-\sigma^\delta_s).
\end{align*}
By the proof of  Theorem~\ref{T:Necc}, there is a constant $C>0$ independent from $\delta$ such that whenever
$\delta< (2CL_b^4)^{-1}\eps\wedge \eps\wedge (T-t)$, we have $(p^\delta,q^\delta)\in\mathbf{O}^t_{\eps,\delta}$.
Thus we can assume from now on that $\delta$ is sufficiently small such that $\varphi$ satisfies
\begin{align*}
\varphi(t,\mathbf{x})+\delta^2&\ge \Mean\left[\varphi(t+\delta,X^{t,\mathbf{x},a^\delta})\right]
=\Mean\left[\varphi(t+\delta,X^{t,\mathbf{x};b^\delta,\sigma^\delta;p^\delta,q^\delta})\right].
\end{align*}
This is possible thanks to $\varphi\in\overline{\Phi}^{\mathcal{E}_+(t,\mathbf{x})}\,V^S(t,\mathbf{x})$ and 
\eqref{E:ViscSuperVeric}.
Now we can proceed as in the proof of Theorem~3.4
in \cite{cosso21v2path} to  deduce that $\inf_{a\in A}\mathcal{L}^a\varphi(t,\mathbf{x})\le 0$.

Next, we deal with the viscosity subsolution property. 
Let $a\in\mathcal{A}^t$ 
 and $\varphi\in\underline{\Phi}^{\mathcal{E}_+^a(t,\mathbf{x})}\,V^S(t,\mathbf{x})$ with corresponding $\eps>0$.
 By the dynamic programming principle,
 \begin{align}\label{E:ViscSubVeric}
 V^S(t,\mathbf{x})\le \Mean\left[V^S(t+\delta,X^{t,\mathbf{x},a})\right]=\Mean\left[V^S(t+\delta,X^{t,\mathbf{x};b,\sigma;p,q})\right]
 \end{align}
 for all $\delta\in(0,T-t]$.
Here, $(b,\sigma)\in\mathcal{E}_+^a(t,\mathbf{x})$ and
  $(p,q)$ is defined as in the proof of Theorem~\ref{T:Necc}, which also yields,
    for all sufficiently small $\delta>0$,
 $(p,q)\in\mathbf{O}^t_{\eps,\delta}$
  and thus, 
  by \eqref{E:ViscSubVeric}, we have
 $ \varphi(t,\mathbf{x})\le 
  \Mean\left[\varphi(t+\delta,X^{t,\mathbf{x};b,\sigma;p,q})\right]$.
Hence $\mathcal{L}^a\varphi (t,\mathbf{x})\ge 0$.

Finally, Theorem~\ref{T:viscosity:comparison} yields uniqueness. This concludes the proof.
\end{proof}

}
\section{The extension lemma}\label{Section:Step3}

Purpose of this section is to prove the 
extension lemma (Lemma~\ref{L:Step3} below), which is needed in Step~3 of the proof of 
Proposition~\ref{P:epsApproxSol}. 

To deal with measurability issues and since we need to use regular conditional probability distributions (r.c.p.d.),
we will work in a weak formulation, i.e., controls will be replaced by probability measures. An appropriate setting
is introduced in Section~\ref{S:Extension:Setting} and the connection between controls and related probability measures
is treated in Section~\ref{S:Extension:Connection}. In Section~\ref{S:Extension:Existence}, we obtain
measurability of controls coming from mean quasi-tangential conditions with respect to initial data. This property is crucial.
Finally, Section~\ref{S:Extension:Lemma} contains statement and proof of the extension lemma.

\subsection{Setting}\label{S:Extension:Setting}
The presentation 
is adapted from \cite{ElKarouiTanII}.

Given a topological space $E$,
we define the
 spaces
\begin{align*}
\mathbb{M}(\R_+\times E)&:=\{\text{all $\sigma$-finite measures on $\mathcal{B}(\R_+\times E)$}\}\text{ and}\\
\mathbb{M}_E&:=\{m\in\mathbb{M}(\R_+\times E):\, m(dt,de)=m(t,de)\,dt\},
\end{align*}
i.e., the marginal distribution on $\R_+$ of any  
element of $\mathbb{M}_E$ equals the Le\-besgue measure.
Note that we 
use the same notation for an element $m(dt,de)$ of $\mathbb{M}_E$ and its kernel $m(t,de)$.

Our extended canonical space is
\begin{align*}
\tilde{\Omega}:=\Omega\times\Omega\times\mathbb{M}_A\times\mathbb{M}_{\R^d}\times\mathbb{M}_{\R^{d\times d}}
\end{align*}
with canonical process $(\tilde{W},\tilde{X},\tilde{M}_A,\tilde{M}_P,\tilde{M}_Q)$ defined by
$\tilde{W}_t(\tilde{\omega}):=\mathring{\omega}_t$, $\tilde{X}_t(\tilde{\omega}):=\omega_t$, 
$\tilde{M}_A(\tilde{\omega}):=m_A$, $\tilde{M}_P(\tilde{\omega}):=m_P$, and $\tilde{M}_Q
(\tilde{\omega}):=m_Q$
for every $t\ge 0$ and every $\tilde{\omega}=(\mathring{\omega},\omega,m_A,m_P,m_Q)\in\tilde{\Omega}$. 
The (raw) filtration generated by this canonical process is denoted by $\tilde{\F}=(\tilde{\cF_t})_{t\ge 0}$.
Given $t\ge 0$, we also use the shifted canonical process 
$(\tilde{W}^t,\tilde{M}_A^t,\tilde{M}_P^t,\tilde{M}_Q^t)$ defined by $\tilde{W}^t_s:=
\tilde{W}_{s+t}-\tilde{W}_t$, $\tilde{M}_A^t(s,da)\,ds:=\tilde{M}_A(t+s,da)\,ds$,
$\tilde{M}_P^t(s,dp)\,ds:=\tilde{M}_P(t+s,dp)\,ds$, and
$\tilde{M}_Q^t(s,dq)\,ds:=\tilde{M}_Q(t+s,dq)\,ds$,
 $s\ge 0$.

Solutions of controlled stochastic differential equations 
are expressed as solutions of martingale problems
(quite similarly as in  \cite{HaussmannLepeltier90SICON}).
To this end,
consider functions 
$\tilde{b}:\R_+\times\Omega\times A\times\R^d\to\R^{2d}$ and
$\tilde{\sigma}:\R_+\times\Omega\times A\times\R^{d\times d}\to\R^{d\times d}\times\R^{d\times d}$ defined by
\begin{align*}
\tilde{b}(t,\mathbf{x},a,p):=(0,b(t,\mathbf{x},a)+p)^\top \text{ and }
\tilde{\sigma}(t,\mathbf{x},a,q):=\left(\begin{array}{c} I_{d\times d} \\ \sigma(t,\mathbf{x},a)+q\end{array}\right).
\end{align*}
Next, we define an infinitesimal generator $\tilde{\mathcal{L}}$ on $C^2_b(\R^{2d})$ as follows.
For each $\varphi\in C^2_b(\R^{2d})$ and each $(t,\mathbf{x},a,p,q,x)\in \R_+\times\Omega\times A\times
\R^d\times\R^{d\times d}\times\R^{2d}$, put
\begin{align*}
(\tilde{\mathcal{L}}\varphi)(t,\mathbf{x},a,p,q,x):=
\tilde{b}(t,\mathbf{x},a,p)\cdot D\varphi(x)+\frac{1}{2}(\tilde{\sigma}\tilde{\sigma}^\top)(t,\mathbf{x},a,q):D^2\varphi(x).
\end{align*}
For each $t$, $t_0\in \R_+$ and $\varphi\in C^2_b(\R^{2d})$, define a random variable $C^{t_0}_t(\varphi)$ on
$\tilde{\Omega}$ by
\begin{equation}\label{E:Ct}
\begin{split}
C^{t_0}_t(\varphi)(\tilde{\omega})&:=\Biggl[\varphi(\tilde{W}_t,\tilde{X}_t)-
\int_0^t \int_A\int_{\R^d}\int_{\R^{d\times d}} 
(\tilde{\mathcal{L}}\varphi)(s,\tilde{X}_{\cdot\wedge t_0},a,p,q,\tilde{W}_s,\tilde{X}_s) \\
&\qquad\qquad
\tilde{M}_Q(s,dq)\,\tilde{M}_P(s,dp)\,\tilde{M}_A(s,da)\,ds\Biggr](\tilde{\omega}) 
\end{split}
\end{equation}
if $\tilde{\omega}=(\mathring{\omega},\omega,\delta_{\bar{a}(s)}(da)\,ds,\delta_{\bar{p}(s)}(dp)\,ds,\delta_{\bar{q}(s)}(dq)\,ds)$ for some
Borel measurable function $(\bar{a},\bar{p},\bar{q}):\R_+\to A\times\R^d\times\R^{d\times d}$
with $\int_0^s \abs{\bar{p}^i(r)}\,dr$ and $\int_0^s \abs{\bar{q}^{i,j}(r)}\,dr$ being finite for each $i$, $j\in\{1,\ldots,d\}$ 
 and each $s\in\R_+$
 (cf.~(1.6) and Remark~1.4, both in \cite{ElKarouiTanII}),
and define $C^{t_0}_t(\varphi)$ by  $C^{t_0}_t(\varphi)(\tilde{\omega}):=0$ otherwise.
We also use the map
\begin{align*}
&\Upsilon: \mathcal{A}\times\mathbb{L}^2(\F,\Prob,\R^d)\times\mathbb{L}^2(\F,\Prob,\R^{d\times d})\to
\mathcal{P}(\Omega\times\mathbb{M}_A\times\mathbb{M}_{\R^d}\times\mathbb{M}_{\R^{d\times d}}),
\\ &\qquad\qquad (\tilde{a},\tilde{p},\tilde{q})\mapsto\Prob\circ
(W,\delta_{\tilde{a}_s(W)}(da)\,ds,\delta_{\tilde{p}_s(W)}(dp)\,ds,\delta_{\tilde{q}_s(W)}(dq)\,ds)^{-1},
\end{align*}
and its image $\Upsilon(\mathcal{A}\times\mathbb{L}^2(\F,\Prob,\R^d)\times\mathbb{L}^2(\F,\Prob,\R^{d\times d}))$,
which is a Borel set 
(Lemma~4.10 in \cite{Djete22AOP} and p.~23 in \cite{ElKarouiTanII}).
{\color{black} Recall that $\mathcal{P}(\ldots)$  and $\delta_{(\cdot)}$ were defined in Section~\ref{S:Notation}.}

\begin{definition}\label{D:strongControlRule}
 Let $(t,\mathbf{x},\mathring{\mathbf{x}})\in [0,T)\times\Omega\times\Omega$. 
 Recall $(a^\circ,p^\circ,q^\circ)$ from \eqref{E:GlobalTriple}.
 We call 
$\tilde{\Prob}\in\mathcal{P}(\tilde{\Omega})$ a \emph{strong control rule starting at
$(t,\mathbf{x},\mathring{\mathbf{x}})$} if,  
$\tilde{\Prob}$-a.s.,
\begin{align*}
&(\tilde{X},\tilde{W},\tilde{M}_A(da),\tilde{M}_{P}(dp),\tilde{M}_{Q}(dq))\vert_{[0,t]}=(\mathbf{x},\mathring{\mathbf{x}},\delta_{a^\circ}(da),\delta_{p^\circ}(dp),\delta_{q^\circ}(dq))\vert_{[0,t]},
\end{align*}
the processes $(C^t_s)(\varphi)_{t\le s\le T_1}$, 
$\varphi\in C^2_b(\R^{2d})$,  $T_1>t$,  are $(\tilde{\Prob},\tilde{\F})$-martingales, and
\begin{equation}\label{E:StrongControl:Inclusion}
\begin{split}
&\tilde{\Prob}\circ(\tilde{W}^t,\tilde{M}^t_A(s,da)\,ds,\tilde{M}^t_P(s,dp)\,ds,\tilde{M}^t_Q(s,dq)\,ds)^{-1}\\&\qquad\qquad \in
\Upsilon ( \mathcal{A}\times\mathbb{L}^2(\F,\Prob,\R^d)\times\mathbb{L}^2(\F,\Prob,\R^{d\times d})).
\end{split}
\end{equation}
The \emph{set of all strong control rules starting at $(t,\mathbf{x},\mathring{\mathbf{x}})$}
is denoted by $\tilde{\mathcal{P}}_{t,\mathbf{x},\mathring{\mathbf{x}}}$.
\end{definition}

\subsection{Connections between ``strong" controls and strong  control rules}\label{S:Extension:Connection}

For every $(t,\mathbf{x},\mathring{\mathbf{x}})\in [0,T)\times\Omega\times\Omega$
and every $(\tilde{a},\tilde{p},\tilde{q})\in
 \mathcal{A}^t\times \mathbb{L}^2(\F^t,\Prob,\R^d)\times\mathbb{L}^2(\F^t,\Prob,\R^{d\times d})$, consider
 the induced probability measure
 \begin{align}\label{E:ProbInduced}
 \Prob_{t,\mathbf{x},\mathring{\mathbf{x}};\tilde{a},\tilde{p},\tilde{q}}:=
 \Prob\circ(W^{t,\mathring{\mathbf{x}}},X^{t,\mathbf{x};\tilde{a},\tilde{p},\tilde{q}},
 \delta_{\tilde{a}_s}(da)\,ds,\delta_{\tilde{p}_s}(dp)\,ds,\delta_{\tilde{q}_s}(dq)\,ds)^{-1}
 \end{align}
 in $\mathcal{P}(\tilde{\Omega})$, where $W^{t,\mathring{\mathbf{x}}}:=\mathring{\mathbf{x}}.\bfone_{[0,t)}
 +(\mathring{\mathbf{x}}_t+W-W_t).\bfone_{[t,\infty)}$
 and
 $X^{t,\mathbf{x};\tilde{a},\tilde{p},\tilde{q}}$ is the unique 
 solution
 of 
 \begin{align}\label{E:6.2:Xapq}
 dX^{t,\mathbf{x};\tilde{a},\tilde{p},\tilde{q}}_s=[b(s,\mathbf{x}_{\cdot\wedge t},\tilde{a}_s)+\tilde{p}_s]\,ds
 +[\sigma(s,\mathbf{x}_{\cdot\wedge t},\tilde{a}_s)+\tilde{q}_s]\,dW_s,\quad\text{$\Prob$-a.s.,}
 \end{align}
 on $[t,\infty)$
 with initial condition $X^{t,\mathbf{x};\tilde{a},\tilde{p},\tilde{q}}\vert_{[0,t]}=\mathbf{x}\vert_{[0,t]}$.
 

 \begin{lemma}\label{L:Ptxxapq:Ptxx}
 $ \Prob_{t,\mathbf{x},\mathring{\mathbf{x}};\tilde{a},\tilde{p},\tilde{q}}
 \in\tilde{\mathcal{P}}_{t,\mathbf{x},\mathring{\mathbf{x}}}$.
 \end{lemma}
 
 \begin{proof}
 We verify only \eqref{E:StrongControl:Inclusion}. Showing the rest is more standard.
First, we define a ``path shifting map" $\iota_t:\Omega\to\Omega$ by
 \begin{align}\label{E:iotaDef}
[ \iota_t(\omega)](\theta):=\omega_0.\bfone_{[0,t)}(\theta)+\omega_{\theta-t}.\bfone_{[t,\infty)}(\theta),\quad \theta\ge 0.
 \end{align}
 Note that there is a Borel map
 $f:\R_+\times\Omega\to A\times\R^d\times\R^{d\times d}$ such that, for all $s\in\R_+$, we have
 $(\tilde{a},\tilde{p},\tilde{q})_s=f_s
  (W_{(\cdot\wedge s)\vee t}-W_t)
 $. Thus, $\Prob$-a.s.,
 for all $s\in\R_+$,
 \begin{align*}
 (\tilde{a},\tilde{p},\tilde{q})_{t+s}&=f_{t+s}
 ((W_{\cdot\wedge (t+s)}-W_t).\bfone_{[t,\infty)})
 =f_{t+s}([\iota_t(W_{t+\cdot}-W_t)]
 (\cdot\wedge (t+s))).
 \end{align*}
 With $\tilde{\Prob}=\Prob_{t,\mathbf{x},\mathring{\mathbf{x}};\tilde{a},\tilde{p},\tilde{q}}$, we therefore have 
 \begin{align*}
 &\tilde{\Prob}\circ(\tilde{W}^t,\tilde{M}^t_A(s,da)\,ds,\tilde{M}^t_P(s,dp)\,ds,\tilde{M}^t_Q(s,dq)\,ds)^{-1}\\
 &\qquad = \Prob\circ(W^{t,\mathring{\mathbf{x}}}_{t+\cdot}-W^{t,\mathring{\mathbf{x}}}_t,
 \delta_{\tilde{a}_{t+s}}(da)\,ds,  \delta_{\tilde{p}_{t+s}}(dp)\,ds, \delta_{\tilde{q}_{t+s}}(dq)\,ds)^{-1}\\
 &\qquad =\Prob\circ(W,\delta_{\tilde{a}^\prime_{s}}(da)\,ds,  \delta_{\tilde{p}^\prime_{s}}(dp)\,ds, 
 \delta_{\tilde{q}^\prime_{s}}(dq)\,ds)^{-1},
 \end{align*}
 where $ (\tilde{a}^\prime,\tilde{p}^\prime,\tilde{q}^\prime)_s:=f_{t+s}([\iota_t(W)]
  (\cdot\wedge (t+s)))$, $s\in\R_+$.
 From \eqref{E:iotaDef},  we can deduce that  $\tilde{a}^\prime$, $\tilde{p}^\prime$, and $\tilde{q}^\prime$
 are $\F$-progressive and thus \eqref{E:StrongControl:Inclusion} holds.
 \end{proof}
 
 \begin{lemma}\label{L:Ptxx:Ptxxapq}
 $\forall \tilde{\Prob}\in\tilde{\mathcal{P}}_{t,\mathbf{x},\mathring{\mathbf{x}}}:\,
 \exists (\tilde{a},\tilde{p},\tilde{q}):
 \,\tilde{\Prob}=\Prob_{t,\mathbf{x},\mathring{\mathbf{x}};\tilde{a},\tilde{p},\tilde{q}}$.
 \end{lemma}
 
 \begin{proof}
 Fix $\tilde{\Prob}\in\tilde{\mathcal{P}}_{t,\mathbf{x},\mathring{\mathbf{x}}}$.
 First, note that, 
by~\eqref{E:StrongControl:Inclusion}, there exists  an 
$(\tilde{a},\tilde{p},\tilde{q})\in \mathcal{A}\times\mathbb{L}^2(\F,\Prob,\R^d)\times\mathbb{L}^2(\F,\Prob,\R^{d\times d})$, such that  
\begin{align*}
&(\tilde{M}_A^t(s,da),\tilde{M}_{P}^t(s,dp),\tilde{M}_{Q}^t(s,dq))=
(\delta_{\tilde{a}_s(\tilde{W}^t)}(da),
\delta_{\tilde{p}_s(\tilde{W}^t)}(dp),
\delta_{\tilde{q}_s(\tilde{W}^t)}(dq)),
 \text{$\tilde{\Prob}$-a.s.}
 \end{align*}
 Next, define $\tilde{\xi}:\R_+\times\tilde{\Omega}\to\R^d$ by  $\tilde{\xi}\vert_{[0,t]}:=\mathbf{x}\vert_{[0,t]}$ and
\begin{align*}
\tilde{\xi}_{t+s}
 &:=\mathbf{x}_t+\int_0^{s} \left[
 b_{t+r}(\mathbf{x}_{\cdot\wedge t},\tilde{a}_r(\tilde{W}^t))
 +\tilde{p}_r(\tilde{W}^t)\right]
 \,dr\\
 &\qquad\qquad 
 +\int_0^{s} \left[
  \sigma_{t+r}(\mathbf{x}_{\cdot\wedge t},\tilde{a}_r(\tilde{W}^t))
 +\tilde{q}_r(\tilde{W}^t)\right]
 \,d\tilde{W}^t_r,\quad\text{$\tilde{\Prob}$-a.s.,}\quad s\ge 0.
\end{align*}
Using $W^t:=W_{t+\cdot}-W_t$, we 
define $(\tilde{a}^\prime,\tilde{p}^\prime,\tilde{q}^\prime):
\R_+\times\Omega\to A\times\R^d\times\R^{d\times d}$
by
\begin{align}\label{E:PrimeControl}
(\tilde{a}^\prime,\tilde{p}^\prime,\tilde{q}^\prime)_s:=(a^\circ,p^\circ,q^\circ).\bfone_{[0,t]}(s)
+(\tilde{a},\tilde{p},\tilde{q})_{s-t}({W}^t).\bfone_{(t,\infty)}(s),\, s\ge 0.
\end{align}
Then 
$(\tilde{a}^\prime,\tilde{p}^\prime,\tilde{q}^\prime)\in
\mathcal{A}^t\times\mathbb{L}^2(\F^t,\Prob,\R^d)\times\mathbb{L}^2(\F^t,\Prob,\R^{d\times d})$ and
\begin{align*}
\tilde{\Prob}\circ (\tilde{W},\tilde{\xi},\tilde{M}_A(s,da)\,ds,\tilde{M}_P(s,dp)\,ds,
\tilde{M}_Q(s,dq)\,ds)^{-1}=\Prob_{t,\mathbf{x},\mathring{\mathbf{x}};\tilde{a}^\prime,\tilde{p}^\prime,\tilde{q}^\prime}.
\end{align*}
It remains to verify  that $\tilde{X}=\tilde{\xi}$, $\tilde{\Prob}$-a.s., 
which can be done by proceeding as in the proof of Lemma~A.1 in \cite{Cox21}, i.e., by applying 
It\^o's formula to $(\lvert\tilde{X}_s-\tilde{\xi}_s\rvert^2)_{s\ge t}$ and utilizing the $(\tilde{\Prob},\tilde{\F})$-martingale
properties of the processes
 $(C^t_s)(\varphi)_{t\le s\le T_1}$, 
$\varphi\in C^2_b(\R^{2d})$, $T_1>t$, from Definition~\ref{D:strongControlRule}. This concludes the proof.
\end{proof}

\subsection{Existence of ``extension" control rules}\label{S:ExtensionControl}\label{S:Extension:Existence}

Fix $\eps>0$. 
Suppose that $(t,\mathbf{x},y)\in [0,T)\times\Omega$  with $\widehat{v}(t,\mathbf{x},y)\in K$
 implies
  $\mathcal{E}_+(t,\mathbf{x})\in\mathcal{QTS}_{\widehat{v},K}(t,\mathbf{x},y)$.  
 Then
 there
is
 a  
 $(\delta,\tilde{a},\tilde{p},\tilde{q})\in (0,\eps]\times
 \mathcal{A}^t\times \mathbb{L}^2(\F^t,\Prob,\R^d)\times\mathbb{L}^2(\F^t,\Prob,\R^{d\times d})$
such that, for all $\mathring{\mathbf{x}}\in\Omega$,
$\Mean^{ \Prob_{t,\mathbf{x},\mathring{\mathbf{x}};\tilde{a},\tilde{p},\tilde{q}}}\left[
\widehat{v}(t+\delta,\tilde{X},y)
 \right]=\Mean \left[\widehat{v}(t+\delta,X^{t,\mathbf{x};\tilde{a},\tilde{p},\tilde{q}},y)\right]\le\eps\delta$ and
\begin{align}\label{E:Mp:Mq}
\Mean^{\tilde{\Prob}}\left[ \int_t^{t+\delta} \left[\int_{\R^d}\abs{p}^2\,\tilde{M}_P(s,dp)+\int_{\R^{d\times d}}\abs{q}^2\,
\tilde{M}_Q(s,dq)\right]\,ds\right]\le  \eps\delta,
\end{align}
where $\tilde{\Prob}=\Prob_{t,\mathbf{x},\mathring{\mathbf{x}};\tilde{a},\tilde{p},\tilde{q}}$. 
Put $\Lambda:=\{(t,\mathbf{x},y)\in [0,T)\times\Omega\times\R:\, 
\widehat{v}(t,\mathbf{x},y)\in K\}$. 
 Then,
taking 
Lemma~\ref{L:Ptxxapq:Ptxx}
into account,  we can see that the following holds:
\begin{equation}\label{E:weakProperty}
\begin{split}
\forall (t,\mathbf{x},y,\mathring{\mathbf{x}})\in 
\Lambda
\times\Omega:
\exists (\delta,\tilde{\Prob})\in (0,\eps]\times\tilde{\mathcal{P}}_{t,\mathbf{x},\mathring{\mathbf{x}}}:
 \Mean^{\tilde{\Prob}}\left[\widehat{v}(t+\delta,\tilde{X},y)\right]\le\eps\delta.
\end{split}
\end{equation}


Consider  
\begin{equation}\label{E:definition:M}
\begin{split}
\mathbf{M}&:=\{ (t,\mathbf{x},y,\mathring{\mathbf{x}},\delta,\tilde{\Prob})
\in 
\Lambda
\times\Omega\times (0,\eps]\times\mathcal{P}(\tilde{\Omega}):\\
&\qquad\qquad
\tilde{\Prob}\in\tilde{\mathcal{P}}_{t,\mathbf{x},\mathring{\mathbf{x}}}, \,
\Mean^{\tilde{\Prob}} [\widehat{v}(t+\delta,\tilde{X},y)]\le\eps\delta,\,\text{ and \eqref{E:Mp:Mq} holds}\}.
\end{split}
\end{equation}

\begin{lemma}\label{L:M_analytic}  Let $\widehat{v}$ be l.s.a. 
Then $\mathbf{M}$
is analytic.
\end{lemma}
\begin{proof}
First, note that, 
by Lemma~7.30~(3) in \cite{BertsekasShreve78book}, 
$(t,\delta,\mathbf{x},y)\mapsto \widehat{v}(t+\delta,\mathbf{x},y)$ is l.s.a.  Hence, by Proposition~7.48 in 
\cite{BertsekasShreve78book},
$(t,\mathbf{x},y,\mathring{\mathbf{x}},\delta,\tilde{\Prob})\mapsto\Mean^{\tilde{\Prob}} [\widehat{v}(t+\delta,\tilde{X},y)]$ is l.s.a.
Thus, together with
 Lemma~7.30~(1) in \cite{BertsekasShreve78book}, the set
$\{ (t,\mathbf{x},y,\mathring{\mathbf{x}},\delta,\tilde{\Prob}):\,
\Mean^{\tilde{\Prob}}[\widehat{v}(t+\delta,\tilde{X},y)]\le \eps\delta\}$ is analytic.
We skip the rest of the proof, which is basically identical to the proof of Lemma~3.2 in \cite{ElKarouiTanII} 
(cf.~also Lemmata~4.5 and 4.10 in \cite{Djete22AOP}).
\end{proof}

Next, define a map 
$H:\Lambda
\times\Omega\times (0,\eps]\times\mathcal{P}(\tilde{\Omega})\to\overline{\R}$ by
\begin{align*}
H(t,\mathbf{x},y,\mathring{\mathbf{x}},\delta,\tilde{\Prob}):=
(+\infty).\bfone_{\mathbf{M}^c}(t,\mathbf{x},y,\mathring{\mathbf{x}},\delta,\tilde{\Prob}).
\end{align*}
Since $\mathbf{M}$ is analytic (Lemma~\ref{L:M_analytic}), 
  $H$ is l.s.a. Define  
 $H^\ast:
 \Lambda
 \times\Omega\to\overline{\R}$ 
by 
\begin{align*}
H^\ast(t,\mathbf{x},y,\mathring{\mathbf{x}}):=\inf_{(\delta,\tilde{\Prob})\in (0, 
\eps]\times\mathcal{P}(\tilde{\Omega})} 
H(t,\mathbf{x},y,\mathring{\mathbf{x}},\delta,\tilde{\Prob}).
\end{align*}
By \eqref{E:Mp:Mq} and \eqref{E:weakProperty},
Proposition~7.50 in \cite{BertsekasShreve78book}  yields the existence of a universally (in fact, even analytically) measurable minimizer
$(\delta^\ast,\mathbb{Q}^\ast): 
\Lambda
\times\Omega\to (0, 
\eps]\times\mathcal{P}(\tilde{\Omega})$, i.e.,
\begin{align*}
H^\ast(t,\mathbf{x},y,\mathring{\mathbf{x}})=
H(t,\mathbf{x},y,\mathring{\mathbf{x}},\delta^\ast_{t,\mathbf{x},y,\mathring{\mathbf{x}}},
\mathbb{Q}^\ast_{t,\mathbf{x},y,\mathring{\mathbf{x}}})=0.
\end{align*}

Note that, since $(t,\mathbf{x},y,\mathring{\mathbf{x}},
\delta^\ast_{t,\mathbf{x},y,\mathring{\mathbf{x}}},
\mathbb{Q}^\ast_{t,\mathbf{x},y,\mathring{\mathbf{x}}})\in\mathbf{M}$ for every
$(t,\mathbf{x},y,\mathring{\mathbf{x}})\in
\Lambda
\times\Omega$, we have, 
by the definitions of $\mathbf{M}$ and $\tilde{\mathcal{P}}_{t,\mathbf{x},\mathring{\mathbf{x}}}$ (see \eqref{E:definition:M} and Definition~\ref{D:strongControlRule}),
\begin{align}\label{E:deltaQ:adapted}
\mathbb{Q}^\ast_{t,\mathbf{x},y,\mathring{\mathbf{x}}}
=
\mathbb{Q}^\ast_{t,\mathbf{x}_{\cdot\wedge t},y,\mathring{\mathbf{x}}_{\cdot\wedge t}}.
\end{align}
\begin{remark}\label{R:Before:L:Step3}
{\color{black} By \eqref{E:structural}, }
$\widehat{v}(t,\mathbf{x},v(t,\mathbf{x}))
\in K$ and thus
\begin{align*}
[0,T)\times\Omega\times\Omega\to(0,\eps 
]\times\mathcal{P}(\tilde{\Omega}),
(t,\mathbf{x},\mathring{\mathbf{x}})\mapsto 
(\delta^{\ast\ast}_{t,\mathbf{x},\mathring{\mathbf{x}}},\mathbb{Q}^{\ast\ast}_{t,\mathbf{x},\mathring{\mathbf{x}}}):=
(\delta^\ast_{\mathfrak{y}},
\mathbb{Q}^\ast_{\mathfrak{y}})\vert_{
\mathfrak{y}=(t,\mathbf{x},v(t,\mathbf{x}),\mathring{\mathbf{x}})},
\end{align*}
is universally measurable  
(Proposition~7.44 in 
\cite{BertsekasShreve78book}). 
\end{remark}
{

{
\subsection{The extension lemma with its proof}\label{S:Extension:Lemma}

\begin{lemma}[the extension lemma]\label{L:Step3}
Let 
$v$ be l.s.a.~and bounded from below. 
If,
for each $(t,\mathbf{x},y)\in [0,T)\times\Omega\times\R$ 
with $\widehat{v}(t,\mathbf{x},y)\in K$, we have
$\mathcal{E}_+(t,\mathbf{x})\in
\mathcal{QTS}_{\widehat{v},K}(t,\mathbf{x},y)$,
then, given 
$\eps>0$ and $(t,\mathbf{x},y)\in [0,T)\times\Omega\times\R$  with  $\widehat{v}(t,\mathbf{x},y)\in K$,
every $\eps$-approximate solution  
$(\tau^\eps,\varrho^\eps,a^\eps,\bar{p},\bar{q},X^\eps)$ 
 of  \eqref{E:mv:CSDE} for $(\widehat{v},K)$  
 starting at $(t,\mathbf{x},y)$ 
 with $\Prob(\tau^\eps<T)>0$ has an ``extension"
 $\mathfrak{s}^{\eps,+}=(\tau^{\eps,+},\varrho^{\eps,+},a^{\eps,+},p^{\eps,+},q^{\eps,+},X^{\eps,+})$ 
 that satisfies the following:
 
 (i)
$\mathfrak{s}^{\eps,+}$ is an $\eps$-approximate solution of  \eqref{E:mv:CSDE} for $(\widehat{v},K)$  starting at 
$(t,\mathbf{x},y)$.

(ii) 
 Outside of a $\Prob$-evanescent set, 
$(\varrho^{\eps,+},a^{\eps,+},p^{\eps,+},q^{\eps,+})_{\cdot\wedge \tau^\eps}=
(\varrho^\eps,a^\eps,\bar{p},\bar{q})_{\cdot\wedge\tau^\eps}$. 

(iii)
 $\tau^\eps\le\tau^{\eps,+}$. 

(iv)
$\Prob(\tau^\eps<\tau^{\eps,+})>0$.
\end{lemma}
\begin{proof}
Fix $\eps>0$, $(t,\mathbf{x},y)\in [0,T)\times\Omega\times\R$ with $\widehat{v}(t,\mathbf{x},y)\in K$,
 and
an $\eps$-approximate solution $(\tau^\eps,\varrho^\eps,a^\eps,\bar{p},\bar{q},X^\eps)$ 
 of  \eqref{E:mv:CSDE} for $(\widehat{v},K)$  starting at $(t,\mathbf{x},y)$ and satisfying $\Prob(E_1)>0$, where
\begin{align*}
E_1:=\{\tau^\eps<T\}.
\end{align*}

The rest of the proof is structured as follows. In Step~1, we identify $(X^\eps,a^\eps,\bar{p},\bar{q})$ with
the solution of a martingale problem specified by a probability measure $\Prob^\eps$ on $\tilde{\Omega}$.
Thanks to the mean quasi-tangency condition 
\begin{align*}
\mathcal{E}_+(\tau^\eps(\omega),X^\eps(\omega))\in
\mathcal{QTS}_{\widehat{v},K}(\tau^\eps(\omega),X^\eps(\omega),
v(\tau^\eps(\omega),X^\eps(\omega))\text{ whenever $\tau^\eps(\omega)<T$,}
\end{align*}
 we obtain from Section~\ref{S:Extension:Existence}
 extension control rules specified by  probability measures. 
This leads to a kernel, whose measurability is shown in Step~2.
In Step~3, we concatenate this kernel with $\Prob^\eps$. This concatenation 
 has appropriate martingale properties (Step~4) and corresponds to a ``candidate" extended solution $\mathfrak{s}^+$
 (Step~5).
In Step~6, we verify that $\mathfrak{s}^+$ is indeed an $\eps$-approximate solution.

\textit{Step~1.} 
To relate $X^\eps$ to the solution of a martingale problem 
(cf.~Section~\ref{S:Extension:Setting}) 
define $(p^\eps,q^\eps)\in\mathbb{L}^2(\F^t,\Prob,\R^d)\times\mathbb{L}^2(\F^t,\Prob,\R^{d\times d})$ by
\begin{align*}
p^\eps_s:&=\bfone_{[t,\tau^\eps]}(s).[b(s,X^\eps_{\cdot\wedge\varrho^\eps_s},a^\eps_s)+\bar{p}_s]
-b(s,X_{\cdot\wedge\tau^\eps}^\eps,a^\eps_s)\quad\text{and}\\
q^\eps_s&:=\bfone_{[t,\tau^\eps]}(s).[\sigma(s,X^\eps_{\cdot\wedge\varrho^\eps_s},a^\eps_s)
+\bar{q}_s]
-\sigma(s,X_{\cdot\wedge\tau^\eps}^\eps,a^\eps_s),\quad s\ge t,
\end{align*}
so that (cf.~condition (v) of Definition~\ref{D:Approx}), $\Prob$-a.s.~on $[t,\infty)$, we have
\begin{align}\label{E:Extension:Xeps:Dynamics}
dX^\eps_{s\wedge\tau^\eps}=[b(s,X^\eps_{\cdot\wedge\tau^\eps},a^\eps_s)+p^\eps_s]\,ds
+[\sigma(s,X^\eps_{\cdot\wedge\tau^\eps},a^\eps_s)+q^\eps_s]\,dW_s.
\end{align}
Fix some $\mathring{\mathbf{x}}\in\Omega$.  Similarly as in~\eqref{E:ProbInduced}, consider the induced probability measure
\begin{align}\label{E:ProbEps}
\Prob^\eps:=\Prob\circ(
W^{t,\mathring{\mathbf{x}}},X^\eps_{\cdot\wedge\tau^\eps}, 
\delta_{a^\eps_s}(da)\,ds,\delta_{p^\eps_s}(dp)\,ds,\delta_{q^\eps_s}(dq)\,ds
)^{-1}\in\mathcal{P}(\tilde{\Omega}).
\end{align}
For every $T_1> t$ and $\varphi\in C^2_b(\R^{2d})$, the process $C(\varphi):\R_+\times\tilde{\Omega}
\to\R$ defined by\footnote{Recall the definition of $C^{t_0}$ for  deterministic times $t_0$ in \eqref{E:Ct}.
The extension for stopping times is done in the usual way.}
\begin{align*}
C_s(\varphi):=C_s^{\tau^\eps\circ\tilde{W}}(\varphi)
\end{align*}
is, 
due to \eqref{E:Extension:Xeps:Dynamics}, 
 a $(\Prob^\eps,\tilde{\F} )$-martingale on $[t,T_1]$.
Also  note that $\Prob^\eps(\tilde{E}_1)>0$, where 
\begin{align*}
\tilde{E}_1:=
\{\tau^\eps\circ\tilde{W}<T\}\subset\tilde{\Omega}.
\end{align*}

{\color{black} 
\textit{Step~2.} 
Using \eqref{E:deltaQ:adapted} and Remark~\ref{R:Before:L:Step3}, we can apply
 the same argument 
 as  in
the paragraph
 before (2.7) 
  in \cite{NutzVanHandel13SPA} (see, in particular, Lemma~2.5 therein) 
  to deduce the existence of 
 an $\tilde{\cF}_{\tau^\eps\circ\tilde{W}}$-measurable function 
 $\tilde{\delta}^\ast:\tilde{\Omega}\to\R_+$ 
 and an 
 $\tilde{\cF}_{\tau^\eps\circ\tilde{W}}$-measurable kernel
   $\tilde{\mathbb{Q}}_\cdot^\ast:\tilde{\Omega}\to\mathcal{P}(\tilde{\Omega})$ 
such that, for 
$\Prob^\eps$-a.e.~$\tilde{\omega}\in\tilde{E}_1$,
\begin{align}\label{E:QstarQstar}
(\tilde{\delta}^\ast_{\tilde{\omega}},\tilde{\mathbb{Q}}^\ast_{\tilde{\omega}})=
(\delta^{\ast\ast}_{\mathfrak{x}},
\mathbb{Q}^{\ast\ast}_{\mathfrak{x}})
\vert_{\mathfrak{x}=
(\tau^\eps\circ\tilde{W}(\tilde{\omega}),\tilde{X}_{\cdot\wedge\tau^\eps\circ\tilde{W}}(\tilde{\omega}),
\tilde{W}_{\cdot\wedge\tau^\eps\circ\tilde{W}}(\tilde{\omega}))}
\end{align}
and that, for $\Prob^\eps$-a.e.~$\tilde{\omega}\in
\tilde{\Omega}\setminus\tilde{E}_1$, 
\begin{equation}\label{E:QstarPeps}
(\tilde{\delta}^\ast_{\tilde{\omega}},\tilde{\mathbb{Q}}^\ast_{\tilde{\omega}})=(0,\Prob^\eps_{\tilde{\omega}}),
\end{equation}
where $\{\Prob^\eps_{\tilde{\omega}}\}_{\tilde{\omega}\in\tilde{\Omega}}$ 
is an r.c.pd.~of $\Prob^\ast$ given $\tilde{\cF}_{\tau^\eps\circ\tilde{W}}$
(cf.~pp.~13 and 25 in \cite{ElKarouiTanII}).

\textit{Step~3.} Following  \cite{HaussmannLepeltier90SICON} (see, in particular, pp.~880-881) and \cite{ElKarouiTanII} 
(see, in particular, p.~11),
we introduce now concatenated probability measures relevant in our context. First,
given $\tilde{\omega}\in\tilde{\Omega}$, consider the $\sigma$-field
\begin{align*}
&(\tilde{\cF})^{\tau^\eps\circ \tilde{W}(\tilde{\omega})}:=\sigma(\tilde{W}_s,\tilde{X}_s,
\tilde{M}_{A,s}(\varphi_A)-\tilde{M}_{A,\tau^\eps\circ \tilde{W}(\tilde{\omega})}(\varphi_A),
\tilde{M}_{P,s}(\varphi_P)\\&\qquad -\tilde{M}_{P,\tau^\eps\circ \tilde{W}(\tilde{\omega})}(\varphi_P),
\tilde{M}_{Q,s}(\varphi_Q)-\tilde{M}_{Q,\tau^\eps\circ \tilde{W}(\tilde{\omega})}(\varphi_Q):\,
s\ge \tau^\eps\circ W^\eps(\tilde{\omega}),\\ &\qquad 
\varphi_A\in C_b(\R_+\times A),\,\varphi_B\in C_b(\R_+\times \R^d),
\text{ and }
\varphi_Q\in C_b(\R_+\times\R^{d\times d})),
\end{align*}
where $M_{A,s}(\varphi)=\int_0^s  \int_A \varphi(r,a)\, M_A(dr,da)$, etc.
Then we denote by
$\delta_{\tilde{\omega}}\otimes_{\tau^\eps\circ\tilde{W}}\tilde{\mathbb{Q}}^\ast_{\tilde{\omega}}$
 the unique probability measure in $\mathcal{P}(\tilde{\Omega})$ such that 
 \begin{align*}
\delta_{\tilde{\omega}}\otimes_{\tau^\eps\circ\tilde{W}}\tilde{\mathbb{Q}}^\ast_{\tilde{\omega}} \,
((\tilde{X},\tilde{W},\tilde{M}_A(da),\tilde{M}_{P}(dp),\tilde{M}_{Q}(dq))\vert_{[0,\tau^\eps\circ\tilde{W}(\tilde{\omega})]}=
 \tilde{\omega}\vert_{[0,\tau^\eps\circ \tilde{W}(\tilde{\omega})]})=1
 \end{align*}
and that $\delta_{\tilde{\omega}}\otimes_{\tau^\eps\circ\tilde{W}}\tilde{\mathbb{Q}}^\ast_{\tilde{\omega}}$
coincides with $\tilde{\mathbb{Q}}^\ast_{\tilde{\omega}}$ on $(\tilde{\cF})^{\tau^\eps\circ \tilde{W}(\tilde{\omega})}$.
Second, denote by
\begin{align*}
\tilde{\Prob}^\ast=\Prob^\eps\otimes_{\tau^\eps\circ\tilde{W}} {\tilde{\mathbb{Q}}^\ast}_{\cdot}
\end{align*}
 the unique probability measure in $\mathcal{P}(\tilde{\Omega})$ such that
$\{\delta_{\tilde{\omega}}\otimes_{\tau^\eps\circ\tilde{W}}\tilde{\mathbb{Q}}^\ast_{\tilde{\omega}}\}_{\tilde{\omega}\in\tilde{\Omega}}$ is an r.c.p.d.~of
$\tilde{\Prob}^\ast$ given $\tilde{\cF}_{\tau^\eps\circ\tilde{W}}$ and that 
$\tilde{\Prob}^\ast$ coincides with $\Prob^\eps$ on 
$\tilde{\cF}_{\tau^\eps\circ\tilde{W}}$. 
In particular, $\tilde{\Prob}^\ast(\tilde{E}_1)>0$.

\textit{Step~4. Martingale properties of $\tilde{\Prob}^\ast$} (cf.~Lemma~3.3 in \cite{ElKarouiTanII}):
For $\Prob^\ast$-a.e.~$\tilde{\omega}\in\tilde{\Omega}$, every $\varphi\in C^2_b(\R^{2d})$, and every
$T_1\ge T$, the process 
$(C^{\tau^\eps\circ\tilde{W}(\tilde{\omega})}_s(\varphi))_{\tau^\eps\circ\tilde{W}(\tilde{\omega})\le s\le T_1}$
is a $(\tilde{\mathbb{Q}}^\ast_{\tilde{\omega}},\tilde{\F})$- and thus also a 
$(\delta_{\tilde{\omega}}\otimes_{\tau^\eps\circ\tilde{W}}\tilde{\mathbb{Q}}^\ast_{\tilde{\omega}},\tilde{\F})$-martingale. By Galmarino's test (see, e.g., Theorem~100 in \cite{DM}),
$\delta_{\tilde{\omega}}\otimes_{\tau^\eps\circ\tilde{W}}\tilde{\mathbb{Q}}^\ast_{\tilde{\omega}}
 (\tau^\eps\circ\tilde{W}=\tau^\eps\circ\tilde{W}(\tilde{\omega}))=1$.
 Hence
 $(C_s(\varphi))_{\tau^\eps\circ\tilde{W}(\tilde{\omega})\le s\le T_1}$
is a $(\delta_{\tilde{\omega}}\otimes_{\tau^\eps\circ\tilde{W}}\tilde{\mathbb{Q}}^\ast_{\tilde{\omega}},\tilde{\F})$-martingale.
Thus, by Theorem~1.2.10 in \cite{SVbook}, 
 $(C_s(\varphi))_{t\le s\le T_1}$
is a 
$(\tilde{\Prob}^\ast,\tilde{\F})$-martingale.}

{\color{black} \textit{Step~5.} Gluing controls together in a measurable way 
as it has been done in the context of volatility control in 
 Section~3 of \cite{NeufeldNutz13EJP}, we show that there exists an
$(X^{\eps,+},\tilde{a}^\ast,\tilde{p}^\ast,\tilde{q}^\ast)\in\mathbb{S}^2(\F,\Prob,\R^d)\times
 \mathcal{A}\times \mathbb{L}^2(\F,\Prob,\R^d)\times\mathbb{L}^2(\F,\Prob,\R^{d\times d})$
that coincides with $(X^\eps,a^\eps,p^\eps,q^\eps)$ on $\llb 0,\tau^\eps\rrb$ and  that satisfies
\begin{align}\label{E:Extension:ProbAst999}
\tilde{\Prob}^\ast=
 \Prob\circ(W^{t,\mathring{\mathbf{x}}},X^{\eps,+},
 \delta_{\tilde{a}^\ast_s}(da)\,ds,\delta_{\tilde{p}^\ast_s}(dp)\,ds,\delta_{\tilde{q}^\ast_s}(dq)\,ds)^{-1}.
\end{align}
First, we  establish the existence of  a `measurable' map
\begin{align*}
\upsilon=(\upsilon_A,\upsilon_P,\upsilon_Q)&:
\tilde{E}_1\to
 \mathcal{A}\times\mathbb{L}^1(\F,\Prob,\R^d)\times\mathbb{L}^2(\F,\Prob,\R^{d\times d}),\\
 &\qquad \tilde{\omega}\mapsto (\tilde{a},\tilde{p},\tilde{q}),
\end{align*}
such that, with $\tilde{\tau}:=\tau^\eps\circ\tilde{W}$,  the measures $\tilde{\mathbb{Q}}^\ast_{\tilde{\omega}}$ satisfy
\begin{align*}
&(\tilde{M}_A^{\tilde{\tau}(\tilde{\omega})}(s,da),\tilde{M}_{P}^{\tilde{\tau}(\tilde{\omega})}(s,dp),\tilde{M}_{Q}^{\tilde{\tau}(\tilde{\omega})}(s,dq))\\&\qquad =
(\delta_{\tilde{a}_s(\tilde{W}^{\tilde{\tau}(\tilde{\omega})})}(da),
\delta_{\tilde{p}_s(\tilde{W}^{\tilde{\tau}(\tilde{\omega})})}(dp),
\delta_{\tilde{q}_s(\tilde{W}^{\tilde{\tau}(\tilde{\omega})})}(dq)),\qquad
 \text{$\tilde{\mathbb{Q}}^\ast_{\tilde{\omega}}$-a.s.}
 \end{align*}
To this end, consider\footnote{
Recall that $\Upsilon(\tilde{a},\ldots)=\Prob\circ (W,\delta_{\tilde{a}_s}(da)\,ds, \ldots)$ and also  keep 
Lemma~\ref{L:Ptxx:Ptxxapq}, \eqref{E:definition:M},  and \eqref{E:QstarQstar} in mind regarding the non-emptiness of $\mathbf{Y}$.
} (cf.~the set $A$ in Step~2 on p.~9 in \cite{NeufeldNutz13EJP}) the Borel set
\begin{align*}
&\mathbf{Y}:=\{
(\tilde{\omega},\tilde{a},\tilde{p},\tilde{q})\in\tilde{E}_1\times
  \mathcal{A}\times\mathbb{L}^1(\F,\Prob,\R^d)\times\mathbb{L}^2(\F,\Prob,\R^{d\times d}):\,\\
&\quad\tilde{\mathbb{Q}}^\ast_{\tilde{\omega}}
  \circ(\tilde{W}^{\tilde{\tau}(\tilde{\omega})},\tilde{M}^{\tilde{\tau}(\tilde{\omega})}_A(s,da)\,ds,
  \tilde{M}^{\tilde{\tau}(\tilde{\omega})}_P(s,dp)\,ds,
  \tilde{M}^{\tilde{\tau}(\tilde{\omega})}_Q(s,dq)\,ds)^{-1}
 =\Upsilon(\tilde{a},\tilde{p},\tilde{q})
\}.
\end{align*}
Then, by Proposition~7.49 in \cite{BertsekasShreve78book}, there exists an analytically measurable
map 
\begin{align*}
\upsilon=(\upsilon_A,\upsilon_P,\upsilon_Q)&:\mathrm{proj}_{\tilde{E}_1}(\mathbf{Y})=\{\tilde{\omega}\in \tilde{E}_1:\,\exists 
(\tilde{a},\tilde{p},\tilde{q}):\,(\tilde{\omega},\tilde{a},\tilde{p},\tilde{q})\in\mathbf{Y}\}\\ &\qquad\to 
\mathcal{A}\times\mathbb{L}^1(\F,\Prob,\R^d)\times\mathbb{L}^2(\F,\Prob,\R^{d\times d})
\end{align*}
whose graph $\{(\tilde{\omega},\upsilon(\tilde{\omega})):\,\tilde{\omega}\in\mathrm{proj}_{\tilde{E}_1}(\mathbf{Y})\}$
is a subset of $\mathbf{Y}$. 
 Note that, by \eqref{E:QstarQstar}, $\tilde{\Prob}^\eps(\tilde{E}_1\setminus \mathrm{proj}_{\tilde{E}_1}(\mathbf{Y}))=0$.
Thus (cf.~the discussion before \eqref{E:QstarQstar}) there is a map
$\tilde{\upsilon}=(\tilde{\upsilon}_A,\tilde{\upsilon}_P,\tilde{\upsilon}_Q)$ from 
$\tilde{E}_1$
 to
 $\mathcal{A}\times\mathbb{L}^2(\F,\Prob,\R^d)\times\mathbb{L}^2(\F,\Prob,\R^{d\times d})$
 that is $\tilde{\cF}_{\tilde{\tau}}$-measurable\footnote{
 This  is needed below for progressive measurability of the processes $\tilde{a}^m$, etc.
 } and that  coincides with $\upsilon$, $\tilde{\Prob}^\eps$-a.e.~on
 $\tilde{E}_1$.\footnote{
 I.e., $\tilde{\upsilon}:\tilde{E}_1\to\mathcal{A}\times\cdots$, $\tilde{\omega}\mapsto \tilde{\upsilon}(
 \tilde{\omega})=\tilde{\upsilon}(\tilde{\omega}_{\cdot\wedge \tilde{\tau}})$ with
 $\tilde{\mathbb{Q}}^\ast\circ(\tilde{W}^{\tilde{\tau}(\tilde{\omega})},\tilde{M}_A^{\tilde{\tau}(\tilde{\omega})}
 (s,da)\,ds,\ldots)^{-1}=
 \Prob\circ (W,\delta_{[\tilde{\upsilon}_A(\tilde{\omega})]_s}(da)\,ds,\ldots)^{-1}
 =\Upsilon(\tilde{\upsilon}(\tilde{\omega}))$.
 }
 Now, for each $m\in\N$, let $\{B^{m,n}\}_{n\in\N}$ be a countable Borel partition of  
 $\mathcal{A}\times\mathbb{L}^2(\F,\Prob,\R^d)\times\mathbb{L}^2(\F,\Prob,\R^{d\times d})$,
 which is understood to be equipped with a suitable metric (cf.~\cite{NeufeldNutz13EJP}).
Under said metric, the diameter of each set $B^{m,n}$ is not to exceed $1/m$.
For each $m$, $n\in\N$, fix $\beta^{m,n}\in B^{m,n}$. For each $m\in\N$, define
$(\tilde{a}^m,\tilde{p}^m,\tilde{q}^m)
:\R_+\times E_1\to A\times\R^d\times\R^{d\times d}$
by\footnote{
With $\tilde{\upsilon}^m:\tilde{E}_1\to\mathcal{A}\times\cdots$ defined by 
$\tilde{\upsilon}^m(\tilde{\omega}):=\sum_{n\in\N} \beta^{m,n}.\bfone_{B^{m,n}}(\tilde{\nu}(\tilde{\omega}))$,
we have
$(\tilde{a}^m,\cdots)(s,\omega)=\bfone_{[\tau^\eps(\omega),\infty)}(s). 
[\tilde{\upsilon}^m((W^{t,\mathring{\mathbf{x}}},\ldots)(\omega)]_{s-\tau^\eps(\omega)}(W^{\tau^\eps(\omega)}(\omega))$.
Note that $\tilde{\upsilon}^m\to\tilde{\upsilon}$ and thus (see below) 
$(\tilde{a}^\dagger,\ldots)(s,\omega)=\bfone_{[\tau^\eps(\omega),\infty)}(s).[\tilde{\upsilon}((W^{t,\mathring{\mathbf{x}}},
\ldots)(\omega)]_{s-\tau^\eps(\omega)}(W^{\tau^\eps(\omega)}(\omega))$.
} (cf.~\eqref{E:PrimeControl})
\begin{align*}
&(\tilde{a}^m,\tilde{p}^m,\tilde{q}^m)(s,\omega):=\bfone_{[\tau^\eps(\omega),\infty)}(s)\,
\sum_{n\in\N} \beta^{m,n}_{s-\tau^\eps(\omega)}
(W^{\tau^\eps(\omega)}(\omega))\\ &\qquad \cdot \bfone_{
B^{m,n}}(\tilde{\upsilon}(\tilde{\omega}))\vert_{
\tilde{\omega}=(W^{t,\mathring{\mathbf{x}}},X^\eps,
\delta_{a^\eps_r}(da)\,dr,
\delta_{p^\eps_r}(dq)\,dr,\delta_{q^\eps_r}(dq)\,dr)(\omega)
}.
\end{align*}
Exactly as on pp.~10 and 11 in \cite{NeufeldNutz13EJP}, we obtain the existence of a limit
$(\tilde{a}^\dagger,\tilde{p}^\dagger,\tilde{q}^\dagger)$ of the sequence $(\tilde{a}^m,\tilde{p}^m,\tilde{q}^m)_m$.
We use now concatenations $\omega\otimes_\tau\omega^\prime\in\Omega$ 
of paths $\omega$, $\omega^\prime\in\Omega$ at $\F$-stopping times $\tau$
as in \cite{NeufeldNutz13EJP}, i.e., 
\begin{align*}
(\omega\otimes_\tau\omega^\prime)_s:=\bfone_{[0,\tau(\omega))}(s).\omega_s+
\bfone_{[\tau(\omega),\infty)}(s).\left[\omega_{\tau(\omega)}+\omega^\prime_{s-\tau(\omega)}\right].
\end{align*}
Thus, given $\omega\in E_1$, we can define
\begin{align*}
(\tilde{a}^\omega,\tilde{p}^\omega,\tilde{q}^\omega):=\lim_m   (\tilde{a}^m,\tilde{p}^m,\tilde{q}^m)(\cdot+\tau^\eps(\omega),\omega\otimes_{\tau^\eps} \cdot):\R_+\times\Omega\to A\times\R^d\times\R^{d\times d}.
\end{align*}
Note that $\tilde{a}^\omega=\tilde{a}^\dagger (\cdot+\tau^\eps(\omega),\omega\otimes_{\tau^\eps} \cdot)$.\footnote{
The left-hand side is used to show that certain laws are equal. 
The (measurable!) right-hand
side is used in the construction of the glued control $\tilde{a}^\ast$ below.
}
Then\footnote{
To see this, note that
\begin{align*}
\tilde{a}^\omega_r(\omega^\prime)&=\tilde{a}^\dagger_{r+\tau^\eps(\omega)}(\omega\otimes_{\tau^\eps}\omega^\prime)\\
&=\bfone_{[\tau^\eps(\omega),\infty)}(r+\tau^\eps(\omega)).
[\tilde{\upsilon}_A((W^{t,\mathring{\mathbf{x}}},\cdots)(\omega))]_{[r+\tau^\eps(\omega)]-\tau^\eps(\omega)}
(W^{\tau^\eps(\omega)}(\omega\otimes_{\tau^\eps}\omega^\prime))\\
&=[\tilde{\upsilon}_A((W^{t,\mathring{\mathbf{x}}},\cdots)(\omega))]_r
((\omega\otimes_{\tau^\eps}\omega^\prime)_{\cdot+\tau^\eps(\omega)}-
(\omega\otimes_{\tau^\eps}\omega^\prime)_{\tau^\eps(\omega)})\\
&=[\tilde{\upsilon}_A((W^{t,\mathring{\mathbf{x}}},\cdots)(\omega))]_r
(\omega^\prime).
\end{align*}
},  for $\Prob$-a.e.~$\omega\in E_1$, 
\begin{align*}
\tilde{\upsilon}((W^{t,\mathring{\mathbf{x}}},X^\eps,
\delta_{a^\eps_s}(da)\,ds,
\delta_{p^\eps_s}(dq)\,ds,\delta_{q^\eps_s}(dq)\,ds)(\omega))=(\tilde{a}^\omega,\tilde{p}^\omega,\tilde{q}^\omega)
\end{align*}
and thus
\begin{align*}
&\Upsilon(\tilde{a}^\omega,\tilde{p}^\omega,\tilde{q}^\omega)
 =
\tilde{\mathbb{Q}}^\ast_{\tilde{\omega}}
  \circ(\tilde{W}^{\tilde{\tau}(\tilde{\omega})},\tilde{M}^{\tilde{\tau}(\tilde{\omega})}_A(s,da)\,ds,
  \tilde{M}^{\tilde{\tau}(\tilde{\omega})}_P(s,dp)\,ds,\\&\qquad\qquad
  \tilde{M}^{\tilde{\tau}(\tilde{\omega})}_Q(s,dq)\,ds)^{-1}\vert_{
\tilde{\omega}=(W^{t,\mathring{\mathbf{x}}},X^\eps,
\delta_{a^\eps_s}(da)\,ds,
\delta_{p^\eps_s}(dq)\,ds,\delta_{q^\eps_s}(dq)\,ds)(\omega)
}.
\end{align*}
Next, define $(\tilde{a}^\ast,\tilde{p}^\ast,\tilde{q}^\ast):\R_+\times\Omega\to A\times\R^d\times\R^{d\times d}$ by
\begin{align*}
(\tilde{a}^\ast,\tilde{p}^\ast,\tilde{q}^\ast)(s,\omega):=
(a^\eps,p^\eps,q^\eps)(s,\omega).\bfone_{[0,\tau^\eps(\omega))}(s)+
(\tilde{a}^\dagger,\tilde{p}^\dagger,\tilde{q}^\dagger)(s,\omega).\bfone_{[\tau^\eps(\omega),\infty)}(s)
\end{align*}
whenever $\omega\in E_1$ and $(\tilde{a}^\ast,\tilde{p}^\ast,\tilde{q}^\ast):=(a^\eps,p^\eps,q^\eps)$ otherwise.
We denote by $X^{\eps,+}$  the solution of 
 \begin{align}
 dX^{\eps,+}_s=[b(s,X^\eps_{\cdot\wedge \tau^\eps},\tilde{a}^\ast_s)+\tilde{p}^\ast_s]\,ds
 +[\sigma(s,X^\eps_{\cdot\wedge \tau^\eps},\tilde{a}^\ast_s)+\tilde{q}^\ast_s]\,dW_s,\quad\text{$\Prob$-a.s.,}
 \end{align}
 on $[t,\infty)$
 with initial condition $X^{\eps,+}\vert_{[0,t]}=\mathbf{x}\vert_{[0,t]}$.
It suffices now
 to show that \eqref{E:Extension:ProbAst999}
holds or, equivalently (recall \eqref{E:QstarPeps}), that \eqref{E:ToShow:LawsEqual} below holds.}
{\color{black} To simplify the notation, we omit from now on
 the $(p,q)$-components from the ``control term" $(a,p,q)$.  Related obvious
modficiations are implicitly used.
Let $n\in\N$, $t\le t_1<\cdots<t_n$, $R_1=(R_1^1,R_1^2,R_1^3)$, $\ldots$, $R_n=(R_n^1,R_n^2,R_n^3)\in
\mathcal{B}(\R^d)\otimes\mathcal{B}(\R^d)\otimes\mathcal{B}(\R)$, $\varphi\in C_b(A)$.
We write $\tilde{M}$ instead of $\tilde{M}_A$. Also put
$\tilde{\mathfrak{X}}_{t_i}:=(\tilde{W}_{t_i},\tilde{X}_{t_i},\tilde{M}_{t_i}(\varphi))$. We will show that
\begin{equation}\label{E:ToShow:LawsEqual}
\begin{split}
&\tilde{\Prob}^\ast(R^{(1)}\cap\tilde{E}_1)\\&\qquad=\Prob\left(\bigcap_{i=1}^n \left\{
W^{t,\mathring{\mathbf{x}}}_{t_i}\in R_i^1,
X^{\eps,+}_{t_i}
\in R^2_i,
\int_0^{t_i} \varphi(\tilde{a}^\ast_s)\,ds\in R^3_i
\right\}\cap E_1\right),
\end{split}
\end{equation}
where $R^{(i_0)}:=\cap_{i=i_0}^n \{\tilde{\mathfrak{X}}_{t_i}\in R_i\}$. 
First, note  that 
(cf.~p.~140 in \cite{SVbook})
\begin{equation}\label{E:ToShow:LawsEqual2}
\begin{split}
&\tilde{\Prob}^\ast(R^{(1)}\cap\tilde{E}_1)=\int_{\tilde{E}_1}\Biggl[
\bfone_{[0,t_1)}(\tilde{\tau}(\tilde{\omega}))\,\tilde{\mathbb{Q}}^\ast_{\tilde{\omega}}(R^{(1)})
\\ &\qquad\qquad+\sum_{i=1}^{n-1}\bfone_{[t_i,t_{i+1})}(\tilde{\tau}(\tilde{\omega}))
\prod_{j=1}^i \bfone_{\{\tilde{\mathfrak{X}}_{t_j}\in R_j\}} (\tilde{\omega})
\cdot\tilde{\mathbb{Q}}^\ast_{\tilde{\omega}}\left(\bigcap_{j=i+1}^n \{\tilde{\mathfrak{X}}_{t_j}\in R_j\}\right)
\\
&\qquad\qquad+\bfone_{[t_n,\infty)}(\tilde{\tau}(\tilde{\omega}))\prod_{i=1}^n \bfone_{\{\tilde{\mathfrak{X}}_{t_i}\in R_i\}}(\tilde{\omega})
\Biggr]\,\Prob^\eps(d\tilde{\omega}).
\end{split}
\end{equation}
 Next, until the end of this paragraph, let  
$\tilde{\omega}=(W^{t,\mathring{\mathbf{x}}}, 
X^\eps,
\delta_{a^\eps_s}(da)ds)(\omega)$
for every $\omega\in\Omega$.
Consider Borel maps $[\xi(\omega)]_{t_j}
:\Omega\to\R^d$ that satisfy 
\begin{align*}
&[\xi(\omega)]_{t_j}(\tilde{W}^{\tilde{\tau}(\tilde{\omega})})=
\tilde{X}_{t_j}=
\tilde{X}_{\tilde{\tau}(\tilde{\omega})}(\tilde{\omega})\\ &
\qquad +\int_0^{t_j-\tilde{\tau}(\tilde{\omega})} \left[
b_{\tilde{\tau}(\tilde{\omega})+r}(\tilde{X}_{\cdot\wedge\tilde{\tau}(\tilde{\omega})}(\tilde{\omega}),
\tilde{a}^\omega_r(\tilde{W}^{\tilde{\tau}(\tilde{\omega})}))
\right]\,dr\\ &\qquad
+\int_0^{t_j-\tilde{\tau}(\tilde{\omega})} \left[
\sigma_{\tilde{\tau}(\tilde{\omega})+r}
(\tilde{X}_{\cdot\wedge\tilde{\tau}(\tilde{\omega})}(\tilde{\omega}),
\tilde{a}^\omega_r(\tilde{W}^{\tilde{\tau}(\tilde{\omega})}))
\right]\,d\tilde{W}^{\tilde{\tau}(\tilde{\omega})}_r,\text{ $\tilde{\mathbb{Q}}^\ast_{\tilde{\omega}}$-a.s.,}
\end{align*}
and\footnote{
Recall the notation $X^{t,\mathbf{x};a}$ in \eqref{E:mv:CSDE}. 
}
 $[\xi(\omega)]_{t_j}(W^{\tilde{\tau}(\tilde{\omega})})=
X_{t_j}^{\tilde{\tau}(\tilde{\omega}),\tilde{X}(\tilde{\omega});\tilde{a}^\dagger}$, 
$\Prob$-a.s.~(the results in Step~4 in this proof 
guarantee\footnote{
For more details, one can proceed as in the the proof of Lemma~\ref{L:Ptxx:Ptxxapq}.
} the existence of such maps).
Let $i\in\{1,\ldots,n-1\}$. Then 
\begin{align*}
&\tilde{\mathbb{Q}}^\ast_{\tilde{\omega}}(R^{(i)})=
\tilde{\mathbb{Q}}^\ast_{\tilde{\omega}}\left(\cap_{j=i}^n \{\tilde{\mathfrak{X}}_{t_j}\in R_j\}\right)\\
&=\tilde{\mathbb{Q}}^\ast_{\tilde{\omega}}\left(\cap_{j=i}^n \{\
\tilde{W}^{\tilde{\tau}({\tilde{\omega}})}_{t_j-\tilde{\tau}(\tilde{\omega})}+\tilde{W}_{\tilde{\tau}(\tilde{\omega})}(\tilde{\omega})
\in R_j^1,\tilde{X}_{t_j}\in R_j^2,\tilde{M}_{t_j} (\varphi)\in R^3_j
\}\right)
\\
&=\tilde{\mathbb{Q}}^\ast_{\tilde{\omega}}\Biggl(\cap_{j=i}^n \Biggl\{\
\tilde{W}^{\tilde{\tau}({\tilde{\omega}})}_{t_j-\tilde{\tau}(\tilde{\omega})}+\tilde{W}_{\tilde{\tau}(\tilde{\omega})}(\tilde{\omega})
\in R_j^1,
[\xi(\omega)]_{t_j}
(\tilde{W}^{\tilde{\tau}(\tilde{\omega})}_{\cdot\wedge (t_j-\tilde{\tau}(\tilde{\omega}))})\in R_j^2,\\ &\qquad\qquad\qquad
\tilde{M}^{\tilde{\tau}(\tilde{\omega})}_{t_j} (\varphi)+\tilde{M}_{\tilde{\tau}(\tilde{\omega})}(\varphi)\in R^3_j
\Biggr\}\Biggr)\\
&=\Prob\Biggl(\cap_{j=i}^n\Biggl\{W_{t_j-\tilde{\tau}(\tilde{\omega})}+\tilde{W}_{\tilde{\tau}(\tilde{\omega})}
(\tilde{\omega})\in R_j^1, 
[\xi(\omega)]_{t_j}
(W_{\cdot\wedge (t_j-\tilde{\tau}(\tilde{\omega}))})\in R_j^2,\\
&\qquad\qquad\qquad \int_0^{t_j-\tilde{\tau}(\tilde{\omega})}\varphi(\tilde{\alpha}_s^{\omega})\,ds
+\tilde{M}_{\tilde{\tau}(\tilde{\omega})}(\varphi)\in R^3_j\Biggr\}\Biggr)\\
&=\Prob\Biggl(
\cap_{j=i}^n\Biggl\{
W^{t,\mathring{\mathbf{x}}}_{t_j}\in R^1_j, 
X^{\eps,+}_{t_j}
\in R^2_j,
\int_0^{t_j} \varphi(\tilde{a}^\ast_s)\,ds\in R^3_j
\Biggr\}
\Biggr\vert\cF_{\tau^\eps}
\Biggr)(\omega) 
\end{align*}
for $\Prob$-a.e.~$\omega\in E_1$.
Plugging the right-hand side above into \eqref{E:ToShow:LawsEqual2} and using \eqref{E:ProbEps}\footnote{
In the present  simplified notation, 
  \eqref{E:ProbEps} is just 
$\Prob^\eps=\Prob\circ(W^{t,\mathring{\mathbf{x}}},X^{t,\mathbf{x};a^\eps},\delta_{a^\eps_s}\,ds)^{-1}$.}
yields
\eqref{E:ToShow:LawsEqual}. We can conclude that \eqref{E:Extension:ProbAst999} holds.
}

\textit{Step~6}. 
Consider
\begin{align}\label{Es:tau:eps+}
\tau^{\eps,+}:=\tau^\eps+\tilde{\delta}^\ast_{
(W^{t,\mathring{\mathbf{x}}}, X^\eps, 
\delta_{a^\eps_s}(da)ds,
\delta_{p^\eps_s}(da)ds,\delta_{q^\eps_s}(da)ds)
},
\end{align}
which is an $\F$-stopping time (Lemma~V.3 in \cite{Subbotina06}).
Define $\varrho^{\eps,+}:\R_+\times\Omega\to\R_+$ by
\begin{align*}
\varrho^{\eps,+}_s:=\varrho^\eps_s.\bfone_{[0,\tau^\eps)}(s)
+\tau^\eps.\bfone_{[\tau^\eps,\tau^{\eps,+})}(s)
+s.\bfone_{[\tau^{\eps,+},\infty)}(s).
\end{align*}
By~\eqref{E:definition:M} and
 \eqref{Es:tau:eps+}, $s-\eps\le\varrho^{\eps,+}_s\le s$ for all $s\in\R_+$. 
Put $(\tilde{\varrho}_s,\tilde{\varrho}^{+}_s,\tilde{\tau}^+):=(\varrho^\eps_s,\varrho^{\eps,+}_s,\tau^{\eps,+})(\tilde{W})$.
Then, for all $y\ge v(t,\mathbf{x})$ and $s\in [t,T]$, 
\begin{equation}\label{Es:FinalInequality}
\begin{split}
&\Mean\left[
v(\varrho^{\eps,+}_s\wedge\tau^{\eps,+},X^{\eps,+})
\right]=\Mean^{\tilde{\Prob}^\ast} \left[
v(\tilde{\varrho}^{+}_s\wedge(\tilde{\tau}+\tilde{\delta}^\ast),\tilde{X})
\right]\\ &\qquad =
\int_{\tilde{E}_1\cap\{\tilde{\tau}^+\le s\}}
\Mean^{\tilde{\mathbb{Q}}^\ast_{\tilde{\omega}}} \left[
v(\tilde{\tau}(\tilde{\omega})+\tilde{\delta}^\ast_{\tilde{\omega}},\tilde{X})
\right]\,
\Prob^\eps(d\tilde{\omega})\\
&\qquad\qquad+\int_{\tilde{E}_1^c\cup(\tilde{E}_1\cap\{\tilde{\tau}^+>s\})}  v(\tilde{\varrho}_s(\tilde{\omega})\wedge\tilde{\tau}(\tilde{\omega}),\tilde{X}
(\tilde{\omega}))\,\Prob^\eps(d\tilde{\omega})\\
&\qquad \le \int_{\tilde{E}_1\cap\{\tilde{\tau}^+\le s\}} \left[v(\tilde{\tau}(\tilde{\omega}),\tilde{X}(\tilde{\omega}))+\eps\tilde{\delta}^\ast_{\tilde{\omega}}\right]
\,\Prob^\eps(d\tilde{\omega})\\
&\qquad\qquad+\int_{\tilde{E}_1^c\cup(\tilde{E}_1\cap\{\tilde{\tau}^+>s\})}  v(\tilde{\varrho}_s(\tilde{\omega})\wedge\tilde{\tau}(\tilde{\omega},
\tilde{X}(\tilde{\omega}))\,\Prob^\eps(d\tilde{\omega})\\
&\qquad = \Mean\left[\bfone_{E_1\cap\{\tau^{\eps,+}\le s\}}.(v(\tau^\eps,X^\eps)
+\eps\cdot(\tau^{\eps,+}-\tau^\eps)
)\right]\\
&\qquad\qquad  +\Mean\left[\bfone_{E_1^c\cup({E}_1\cap\{\tau^{\eps,+}> s\})}.(v(\varrho^\eps_s\wedge\tau^\eps,X^\eps)\right]\\
&\qquad =\Mean\left[ v(\varrho^\eps_s\wedge\tau^\eps,X^\eps)+\eps\cdot(\tau^{\eps,+}-\tau^\eps).\bfone_{E_1\cap\{\tau^{\eps,+}\le s\}}\right]\\
&\qquad \le y+\eps\cdot\Mean[s\wedge\tau^\eps-t]+\eps\cdot\Mean[(\tau^{\eps,+}-\tau^\eps).\bfone_{\{\tau^{\eps,+}\le s\}}], 
\end{split}
\end{equation}
i.e., $\Mean [\widehat{v}(\varrho^{\eps,+}_s\wedge\tau^{\eps,+},X^{\eps,+},y)]\le \eps\cdot\Mean[s\wedge \tau^{\eps,+}-t]$.
The first inequality in \eqref{Es:FinalInequality} follows from
$\widehat{v}(\tilde{\tau}(\tilde{\omega}),\tilde{X}(\tilde{\omega}),
v(\tilde{\tau}(\tilde{\omega}),\tilde{X}(\tilde{\omega}))
)\in K$,
which (as assumed in the statement of our lemma) implies 
\begin{align*}
\mathcal{E}_+(\tilde{\tau}(\tilde{\omega}),\tilde{X}(\tilde{\omega}))\in
\mathcal{QTS}_{\widehat{v},K}(\tilde{\tau}(\tilde{\omega}),\tilde{X}(\tilde{\omega}),
v(\tilde{\tau}(\tilde{\omega}),\tilde{X}(\tilde{\omega}))),
\end{align*}
together with the definitions of  $\tilde{\mathbb{Q}}^\ast_{\tilde{\omega}}$,   
$\mathbb{Q}^{\ast\ast}$, and $\mathbf{M}$
in  \eqref{E:QstarQstar},  Remark~\ref{R:Before:L:Step3}, 
and \eqref{E:definition:M}.

Define $(p^{\eps,+},q^{\eps,+}):\R_+\times\Omega\to\R^d\times\R^{d\times d}$ by
\begin{align*}
(p^{\eps,+},q^{\eps,+})_s:=\bfone_{[0,\tau^\eps)}(s).(\bar{p},\bar{q})_s+
\bfone_{[\tau^\eps,\tau^{\eps,+})}(s).(\tilde{p}^\ast,\tilde{q}^\ast)_s.
\end{align*}
By \eqref{E:tptq:eps} for $(\tau^\eps,\bar{p},\bar{q})$
and \eqref{E:Mp:Mq} with \eqref{E:definition:M}, we have 
 \eqref{E:tptq:eps}  for $(\tau^{\eps,+},p^{\eps,+},q^{\eps,+})$.

We can conclude that $(\tau^{\eps,+},\varrho^{\eps,+},\tilde{a}^\ast,p^{\eps,+},q^{\eps,+},X^{\eps,+})$
is our desired ``extended" $\eps$-approximate solution.
\end{proof}

\bibliographystyle{amsplain}
\bibliography{K20R1}

\end{document}